\documentclass[11pt]{amsart}
\usepackage{amsbsy,amssymb,amscd,amsfonts,latexsym,amstext,delarray, amsmath,graphicx,color,caption, amsthm, enumerate, verbatim}
\usepackage{amstext,amsopn, mathtools, bbm}
\usepackage{graphicx}
\usepackage{mathtools}
\usepackage{hyperref}
\usepackage{enumitem}
\usepackage{epsfig}
\input xy

\newcommand\Z{{\mathbb Z}}
\newcommand\R{{\mathbb R}}

\newcommand\T{{\mathbb T}}

\hypersetup{
	colorlinks=true,
	linkcolor=blue,
	filecolor=blue,      
	urlcolor=blue,
	citecolor = blue
}

\begin{document}
	
	\newtheorem{example}{Example}[section]
	\newtheorem{lemma}{Lemma}[section]
	\newtheorem{thm}{Theorem}
	\newtheorem{prop}[lemma]{Proposition}
	\newtheorem{cor}{Corollary}[section]
	
	\theoremstyle{remark}
	\newtheorem{remark}{Remark}[section]

	\title[Global Semiperiodic Strichartz Estimates]{On Global-in-Time Strichartz Estimates for the Semiperiodic Schr\"{o}dinger Equation}
	\author{Alex Barron}

	\maketitle
	
	\newcommand{\Addresses}{{
		\bigskip
		\footnotesize

		\textsc{Department of Mathematics, University of Illinois Urbana-Champaign,
			Urbana, IL 61801, USA}\par\nopagebreak
		\textit{E-mail address}: \texttt{aabarron@illinois.edu}

}}

\begin{abstract} We prove global-in-time Strichartz-type estimates for the Schr\"{o}dinger equation on manifolds of the form $\R^{n}\times \T^{d}$, where $\T^{d}$ is a $d$-dimensional torus. Our results generalize and improve a global space-time estimate for the Schr\"{o}dinger equation on $\R \times \T^{2}$ due to Z. Hani and B. Pausader. As a consequence we prove global existence and scattering in $H^{\frac{1}{2}}$ for small initial data for the quintic NLS on $\R\times \T$ and the cubic NLS on $\R^{2} \times \T$. 
\end{abstract}
	
\section{Introduction}

Over the last few decades there has been a wide range of research concerning Strichartz estimates for dispersive partial differential equations on a variety of manifolds, in particular on manifolds which are in an appropriate sense ``smaller'' than $\R^n$. It is of interest to understand the effect of the underlying geometry on the evolution of solutions to dispersive equations on these manifolds, and Strichartz-type space-time estimates are often a useful tool in this direction (see for example \cite{B}, \cite{BGH}, \cite{BGT}, \cite{HP} and the references therein; there is also related work studying equations with variable coefficients, for example in \cite{GS}). A particular case of interest is the Schr\"{o}dinger equation on any manifold containing trapped geodesics, for example on a $d$-dimensional torus $\T^{d}$. The study of Strichartz estimates for this equation dates back to work of Bourgain \cite{B}, and it is only very recently that the full range of (essentially sharp) local $L_{t,x}^{p}$ estimates have been proved as a consequence of Bourgain and Demeter's $\ell^{2}$ decoupling theorem \cite{BD}. See also the work of Killip and Visan \cite{KV}, which sharpens Bourgain and Demeter's Strichartz estimate. 

In this paper we focus on the setting of product manifolds of the form $\R^{n} \times \T^{d}$, where $\T^{d}$ is a (rational or irrational) $d$-dimensional torus. There has been recent interest in the behavior of solutions to the linear and nonlinear Schr\"{o}dinger equation on these manifolds (see for example \cite{CGZ},\cite{HP},\cite{HPTV},\cite{HTT2}, \cite{IP},  \cite{TV}). In particular, one can exploit dispersive effects coming from the Euclidean component of the manifold to obtain stronger asymptotic results than in the setting of $\T^{d}$. Indeed, as a starting point one can hope to prove global-in-time Strichartz-type estimates for solutions to the linear equation on $\R^{n} \times \T^{d}$ (\cite{HP}). This contrasts the situation on $\T^{d}$, where no global $L_{t,x}^{p}$ estimates are possible. One can also hope to prove stronger results than in the more general setting of $\R^{n} \times M^{d}$, where $M^{d}$ is a $d$-dimensional Riemannian manifold, since the presence of the torus facilitates Fourier-analytic and number-theoretic methods in the vein of \cite{B}, \cite{BD}. 

The particular case of the quintic nonlinear Schr\"{o}dinger equation on $\R\times \T^{2}$ has been studied in depth by Hani and Pausader \cite{HP} (see also the other papers cited above). A starting point for their analysis is the following global space-time estimate for the linear flow: \begin{equation}\label{eq:HPstrichartz}\| e^{it \Delta_{\R \times \T^{2}}} P_{\leq N} u_0 \|_{\ell^{q}_{\gamma}L^{p}(\R \times \T^{2} \times [\gamma, \gamma +1] ) } \lesssim N^{ \frac{3}{2} - \frac{5}{p}} \|u_0\|_{L^{2}(\R \times\T^{2}) }\end{equation} whenever $$p > 4 \ \ \ \text{ and  } \ \ \frac{2}{q} + \frac{1}{p} = \frac{1}{2},$$ where $P_{\leq N}$ is a smooth Littlewood-Paley frequency cut-off to scale $N \geq 1$. Notice that the norm \eqref{eq:HPstrichartz} distinguishes between the local-in-time and global-in-time integrability of the solution. Moreover, the exponent $q$ is exactly the Strichartz-admissible time endpoint for $L_{t}^{q}L^{p}_{x}$ estimates for the Schr\"{o}dinger equation on $\R$, while the loss in $N$ is the same as on $\T^{3}$ or $\R^{3}$ (and is in fact the best one can hope for). However, from the theory on $\R^{3}$ or $\T^{3}$ we expect to be able to push the exponent $p$ down to values larger than the three-dimensional Stein-Tomas endpoint $\frac{10}{3}$ (or equal to $\frac{10}{3}$ with a possible loss of $N^{\epsilon}$).

In this paper we show that \eqref{eq:HPstrichartz} is indeed true for $p > \frac{10}{3}$, and also true for $p =\frac{10}{3}$ with an arbitrarily small loss in the power of $N$. We also extend this result to higher (and lower) dimensions and prove the scale-invariant analogue of \eqref{eq:HPstrichartz} on $\R^{n}\times \T^{d}$ for $p$ away from the Stein-Tomas endpoint $\frac{2(n+d+2)}{n+d}$. The main theorem is the following.

\begin{thm} \label{mainThm} Let $\T^{d}$ be a $d$-dimensional rational or irrational torus and $\Delta_{\R^{n} \times \T^{d}}$ the Laplacian on $\R^{n} \times \T^{d}$. Let $p^{\ast} = \frac{2(n+d+2)}{n+d}$ and fix $p > p^{\ast}$. Then if $q = q(p) :=  \frac{4p}{n(p - 2)}$ and $q > 2$, \begin{equation}\label{eq:mainThmEst}\bigg(\sum_{\gamma \in \Z} \| e^{it\Delta_{\R^{n} \times \T^{d} }} u_{0} \|^{q}_{L^{p} (\R^{n} \times \T^{d} \times [\gamma - 1, \gamma +1]) }\bigg)^{1/q}  \lesssim \|u_{0}\|_{H^{s}(\R^{n} \times \T^{d})},\end{equation} where $s = \frac{n+d}{2} - \frac{n + d + 2}{p}$. Moreover, if $p = p^{\ast}$ then the result holds with $q = q(p^{\ast})$ for any $s > 0$ (with a constant that blows up as $s \rightarrow 0$).    \end{thm}

\noindent Note that when $p = p^{\ast}$ we have $q(p) = \frac{2(n+d+2)}{n}$. Setting $d = 0$ or $n = 0$, we recover the usual $L_{t,x}^{p}$ Strichartz estimates on $\R^{n}$ and $\T^{d}$, respectively (modulo a loss of $N^{\epsilon}$ in the Euclidean case). Also note that the time exponent $q(p)$ is exactly the admissible $q$ value corresponding to the $L_{t}^{q}L_{x}^{p}$ Strichartz estimates on $\R^{n}$, which is expected since we heuristically have $n$ directions contributing to dispersion. Indeed, one cannot prove an estimate of type \eqref{eq:mainThmEst} for $q < q(p)$ (see Remark \ref{rmk:qsharp} for a proof). We remark that in the endpoint case $p = p^{\ast}$ one expects the result to be true with $s = 0$, at least when $n \geq d$. However it does not appear that the techniques present in this paper will be enough to tackle the endpoint case. We say a little more about this at the end of Section 6. 

The choice of the space-time norm in \eqref{eq:HPstrichartz} and \eqref{eq:mainThmEst} raises a few natural questions which we briefly address. The first concerns the use of local norms of the form $L^{p}(\R^n \times \T^d \times [\gamma - 1, \gamma +1 ])$. Here it is important for our proof to have time gaps which are bounded above and below (independent of the frequency $N$), though otherwise the choice of $[\gamma- 1, \gamma + 1]$ is arbitrary. This is mainly because our proof makes use of the $\ell^2$ decoupling theorem of Bourgain and Demeter which imposes these restrictions (with constants ultimately depending on the length of the time interval). At present there are no known $\ell^2$ decoupling theorems involving local-in-time mixed norms of the form $L_{t}^{r}L_{x}^{p},$ though theorems of this type would likely yield more general mixed-norm extensions of \eqref{eq:mainThmEst}. We also do not know in general if it is possible to prove a stronger global Strichartz estimate of the form $$\|e^{it \Delta_{\R^n \times \T^d }}f \|_{L_{t}^{q}L_{x,y}^{p} (\R_t \times \R^n \times \T^d)} \leq C \|f\|_{H^s (\R^n \times \T^d)}$$ with optimal scaling for $p$ near the Stein-Tomas endpoint in dimension $n+d$. Note however that it is relatively easy to prove such an estimate for large $p$. For example on $\R \times \T$ one can use Sobolev embedding and the $L_{x,t}^{6}(\R \times \R)$ Strichartz estimate to obtain $$\|e^{it \Delta_{\R \times \T }}f \|_{L_{t}^{6}L_{x,y}^{6} (\R_t \times \R \times \T)} \leq C \|f\|_{H^{\frac{1}{3}} (\R \times \T)}.$$ Here $s = \frac{1}{3}$ is optimal when $p = 6,$ but unfortunately the loss of $\frac{1}{3}$ derivatives is too large for many applications (in particular it is too large to use in the proof of the first part of Theorem \ref{nonlinearThm} in Section 5 below). There are some other interesting related questions which we summarize in Section 6.

Another natural question concerns the precise dependence of an estimate of type \eqref{eq:mainThmEst} on the derivatives in the $x$ and $y$ variables separately. We have the following corollary of the main theorem, which essentially follows from an observation of Hani and Pausader \cite{HP}. 
\begin{cor}\label{cor:mixedDeriv} Suppose $1 \leq M \leq N$ and let $P_{\leq M}^{x}$ denote an $n$-dimensional Littlewood-Paley cut-off to scale $M$ in the $x$ frequency. Then for any $p > \frac{2(n+ d+ 2)}{n+d}$ there is $\theta > 0$ such that $$ \|P_{\leq N} P_{\leq M}^{x} e^{it\Delta_{\R^{n} \times \T^{d} }} u_{0} \|_{\ell_{\gamma}^q L^{p} (\R^{n} \times \T^{d} \times [\gamma - 1, \gamma +1]) } \lesssim \bigg( \frac{M}{N} \bigg)^{\theta} N^{\frac{n+d}{2} - \frac{n + d + 2}{p}}\|u_{0}\|_{L^{2}(\R^{n} \times \T^{d})}.$$
	
\end{cor}

\noindent The corollary follows from the proof of Theorem \ref{mainThm} via an argument similar to the proof of Remark 3.5 in \cite{HP}. We discuss the details at the end of Section 4.  

\subsection{Overview of the rest of the paper} The proof of Theorem \ref{mainThm} begins in Section 3, following a brief review of some preliminary material in Section 2. Our argument combines the approach of Hani and Pausader with the decoupling method of Bourgain and Demeter (see the next section for a precise statement of their decoupling theorem). We will initially prove \eqref{eq:mainThmEst} in the case $u_0 = P_{\leq N} u_0$ with an extra loss of $N^{\epsilon}$, but we show in Section 4 that this loss can be removed away from the Stein-Tomas endpoint $p^{\ast}$. In Section 5 we study some applications of Theorem \ref{mainThm} to the nonlinear theory of Schr\"{o}dinger equations on $\R^{n} \times \T^{d}$. Our two main results in this direction are contained in the following theorem. Recall that the quintic NLS is the equation $$-i\partial_{t}u + \Delta_{\R^{n} \times \T^{d}} u = \pm |u|^{4}u, \ \ \ \ \ \ u(x,y,0) = u_0 (x,y)$$ and the cubic NLS is the equation $$-i\partial_{t}u + \Delta_{\R^{n} \times \T^{d}} u = \pm |u|^{2}u, \ \ \ \ \ \ u(x,y,0) = u_0 (x,y),$$ and that these equations are $H^{\frac{1}{2}}$ critical on $\R \times \T$ and $\R^{2} \times \T$, respectively. Below $X^{\frac{1}{2}}$ is a Banach space of functions $u: \R \rightarrow H^{\frac{1}{2}}$ defined in Section 5, with the property that $X^{\frac{1}{2}} \hookrightarrow L^{\infty}(\R , H^{\frac{1}{2}})$. 

\begin{thm} \label{nonlinearThm} The quintic NLS on $\R \times \T$ with initial data $u_0 \in H^{\frac{1}{2}}(\R \times \T)$ is locally well-posed. Moreover, there exists $\delta > 0$ such that if $\|u_0\|_{H^{\frac{1}{2}}} < \delta$ then the solution $u\in X^{\frac{1}{2}}$ is unique, exists globally in time, and scatters as $t\rightarrow \pm \infty$ in the sense that there are $v_{\pm}\in H^{\frac{1}{2}}$ such that $$\lim_{t\rightarrow \pm \infty}\|u - e^{it\Delta_{\R \times \T}}v_{\pm}\|_{H^{\frac{1}{2}}} = 0.$$ 
	
The cubic NLS on $\R^{2} \times \T$ with initial data $u_0 \in H^{\frac{1}{2}}(\R^{2} \times \T)$ is also locally well-posed. Moreover, there exists $\delta > 0$ such that if $\|u_0\|_{H^{\frac{1}{2}}} < \delta$ then the solution $u\in X^{\frac{1}{2}}$ is unique, exists globally in time, and scatters as $t\rightarrow \pm \infty$ 
\end{thm}

Once we have proved Theorem \ref{mainThm} we can prove Theorem \ref{nonlinearThm} by using some standard machinery which we review at the beginning of Section 5. In this $H^{\frac{1}{2}}$-critical setting it is essential to have an estimate of type \eqref{eq:mainThmEst} without any extra loss of $N^{\epsilon}$. In particular the result in Proposition \ref{mainThmEps} below, which has a shorter proof than Theorem \ref{mainThm}, is not sufficient to establish Theorem \ref{nonlinearThm}. We note that the local-in-time results in Theorem \ref{nonlinearThm} do not require global Strichartz estimates; we will only use the full strength of Theorem \ref{mainThm} to establish the small-data global existence and scattering results.  

Finally, in Section 6 we collect some additional remarks and related open problems.

\subsection{Notation and Basic Assumptions} As is standard, we write $A \lesssim B$ if there is some constant $c > 0$ depending only on the dimension and various Lebesgue exponents such that $A \leq cB$. If $A\lesssim B$ and $B \lesssim A$ we will write $A \sim B$. Moreover, if $A \leq c(\alpha)B$ where the constant $c(\alpha)$ depends on some parameter $\alpha$ we will write $A \lesssim_{\alpha} B$. We will also often write $A \lesssim_{\epsilon} N^{\epsilon}B$ as short-hand for the expression `for all $\epsilon > 0$ there is $c_{\epsilon}$ such that $A \leq c_{\epsilon}N^{\epsilon}B,$' to avoid having to write expressions involving constant multiples of an arbitrarily small parameter $\epsilon > 0$. 

 Let $S$ be a rectangle and let $S^{-1}$ denote the dual rectangle centered at the origin obtained by inverting the side lengths. We let $w_{S}$ be a weight adapted to $S$ in the following sense: $w_{S}(x)$ decays rapidly for $x \notin S$, and $\widehat{w_{S}}(\xi)$ is supported in a fixed dilate of $S^{-1}$. We similarly define $w_{S}$ if $S$ is a ball, and if $\Omega = \bigcup_{S} S$ then we let $w_{\Omega} = \sum_{S}w_{S}$. Note that we can construct $w_{S}$ by taking a bump function $w$ adapted to the unit ball such that $$|w(x)| \lesssim \frac{1}{(1 + |x|)^{1000(n+d)}}$$ and then applying a suitable affine transformation. If $c > 0$ and $S$ is a ball or rectangle, we also let $cS$ denote the centered dilate by $c$. 
 
 For a dyadic integer $N \geq 1$ we let $P_{\leq N}f = f \ast \varphi_{N}$, where $\varphi_{N}$ is a smooth function such that $\widehat{\varphi_N}$ is supported in $B_{2N}(0)$. We also let $P_{N}$ denote a smooth Littlewood-Paley cut-off to the annulus $A_{N} = \{ \xi : N/2 \leq |\xi| < N\}$.  
 
 Finally, we assume that all functions are smooth and rapidly decaying. We can do so with no loss of generality as long as our estimates are independent of the smoothness and decay parameters. 

\subsection{Acknowledgments} The author thanks Benoit Pausader for suggesting the problem that led to this paper, and for several helpful discussions. The author also thanks Jill Pipher and Linhan Li for many helpful discussions, and an anonymous referee for suggestions that have improved the presentation of the paper.   

\section{Preliminaries}

We will study the linear semiperiodic Schr\"{o}dinger equation \begin{equation}\label{eq:Schrodinger}\begin{cases}
-i\partial_{t}u + \Delta_{\R^{n} \times \T^{d}} u = 0,& \\
u(x,y,0) = f(x,y)              & x\in \R^{n}, y \in \T^{d}. 
\end{cases}\end{equation}

\noindent We assume that our initial data $f$ is smooth and rapidly decaying, which we can do with no loss of generality as long as our estimates do not depend on the smoothness of $f$. For most of our arguments in the following sections we will also assume that $\T^{d}$ is the flat torus $\T^{d} \sim [0,1]^{d}$. In this case solutions to \eqref{eq:Schrodinger} can be represented by $$e^{it \Delta_{\R^{n} \times \T^{d}}} f(x,y) = \sum_{m \in \Z^{d}} \int_{\R^{n}} \widehat{f_{m}}(\xi) e^{2\pi i (x \cdot \xi + y \cdot m + t(|\xi|^{2} + |m|^{2} ) } d\xi,$$ where $$f(x,y) = \sum_{m\in \Z^{d}} f_{m}(x) e^{2\pi i y\cdot m}.$$ In the more general setting of an irrational torus $\T^{d} \sim \prod_{i=1}^{d} [0, \beta_{i}]$ we have a similar representation formula obtained by replacing $|m|^{2}$ above by $\sum_{i} \beta_{i}^{2} m_{i}^{2}$. We will explain when necessary how to adapt our arguments to the case $\T^{d} \sim \prod_{i=1}^{d} [0, \beta_{i}].$ 

Recall that solutions to \eqref{eq:Schrodinger} in the Euclidean case $d = 0$ admit the following global estimates (see for example \cite{T}).

\begin{prop}[Euclidean Strichartz estimates] Suppose $\frac{2}{q} + \frac{n}{p} = \frac{n}{2}$. Then $$\|e^{it\Delta_{\R^{n}}}f\|_{L_{t}^{q}L_{x}^{p}} \lesssim \|f\|_{L^{2}(\R^{n})}$$ whenever $(q, p, n) \neq (2, \infty, 2)$.  \end{prop} \noindent We will in particular exploit this estimate when $p = p^{\ast} = \frac{2(n+d+2)}{n+d}$ and $q = \frac{2(n+d+2)}{n}.$

Another important tool will be the $\ell^{2}$ decoupling theorem of Bourgain and Demeter.

\begin{thm}[\cite{BD}] \label{thm:BDparabola} Let $N \geq 1$ and suppose $f$ is a smooth function on $\R^{n+1}$ such that $\widehat{f}$ is supported in an $O(N^{-2})$ neighborhood of the truncated paraboloid $$P^{n} = \{(\xi, \tau) \in \R^{n} \times \R : \tau = |\xi|^{2}, \ |\xi| \leq 1 \}.$$ Let $\{\theta\}$ be a finitely-overlapping collection of $N^{-1}$-caps covering the support of $\widehat{f}$, let $\{\varphi_{\theta} \}$ be a partition of unity subordinate to this cover, and define $f_{\theta} = f \ast \check{\varphi_{\theta}}$. Then for any ball $B_{N^{2}}$ of radius $N^{2}$ in $\R^{n+1}$ one has $$\|f\|_{L^{p}(w_{B_{N^{2}}})} \lesssim_{\epsilon} N^{\epsilon + \alpha_p} \bigg( \sum_{\theta} \|f_{\theta}\|_{L^{p}(w_{B_{N^{2}}}) }^{2} \bigg)^{\frac{1}{2}} $$ for $p \geq 2$, where $\alpha_{p} = 0$ if $2 \leq p \leq \frac{2(n+2)}{n}$ and $\alpha_p = \frac{n}{2} - \frac{n + 2}{p}$ otherwise. 
\end{thm}

Finally, the following discrete analogue of the classical Hardy-Littlewood-Sobolev inequality will also be useful. 

\begin{prop}[Discrete Hardy-Littlewood-Sobolev] \label{prop:discreteHLS} Suppose $1 < p,q < \infty$ and $0 < \mu < 1$ such that $$\frac{1}{p} + \frac{1}{q} + \mu = 2.$$ Then $$\sum_{j \neq k} \frac{a_{j}b_{k}}{|j-k|^{\mu}} \lesssim \|a\|_{\ell^{p}}\|b\|_{\ell^{q}}.$$
\end{prop}

\section{The Strichartz Estimate}

We begin by proving Theorem \ref{mainThm} with a loss of $N^{\epsilon}$. 

\begin{prop}\label{mainThmEps} For all $p \geq p^{\ast}$ and $q \geq  \frac{4p}{n(p - 2)}$ with $q > 2$ we have $$\bigg(\sum_{\gamma \in \Z} \| e^{it\Delta_{\R^{n} \times \T^{d}} }P_{\leq N}f \|^{q}_{L^{p} (\R^{n} \times \T^{d} \times [\gamma -1, \gamma +1]) }\bigg)^{1/q}  \lesssim_{\epsilon} N^{\epsilon + \frac{n+d}{2} - \frac{n + d + 2}{p}}\|f\|_{L^{2}(\R^{n} \times \T^{d})}.$$ \end{prop}

\noindent  We show in Section 4 that this loss of $N^{\epsilon}$ can be removed for $p$ away from the endpoint. 

Since our arguments use the $\ell^{2}$ decoupling theorem to handle the torus component of the manifold, the proposition is true regardless of whether $\T^{d}$ is the standard torus or the irrational torus. We include some more details in comments below.

\subsection{Main Lemmas} In this subsection we collect some preliminary results. For the rest of the section we will write $\Delta = \Delta_{\R^{n} \times \T^{d}}$.

\begin{lemma} \label{decLemma} Suppose $g, g_{l}$ are Schwartz functions on $\R^{n} \times \T^{d}$ with $g = P_{\leq N}g$, such that $$g(x,y) = \sum_{\substack{ l \in \Z^{d} \\ |l|\leq N}} \int_{[-N,N]^{n}} \widehat{g}_{l}(\xi)e^{2\pi i (x\cdot \xi + y\cdot l)} d\xi.$$ Cover $[-N,N]^{n}$ by finitely-overlapping cubes $Q_{k}$ of side-length $\sim 1$, let $\{\phi_k\}$ be a partition of unity adapted to the $\{Q_k \}$, and define $g_{\theta_{m,k}} = e^{2\pi i y\cdot m}\mathcal{F}_{x}^{-1}(\widehat{g}_{m} \phi_{k}).$ Also let $p^{\ast} =  \frac{2(n+d+2)}{n+d}$ and $p \geq p^{\ast}$. Then for any time interval $I$ of length $\sim 1$ we have  $$\|e^{it\Delta}g\|_{L^{p}(\R^{n} \times \T^{d} \times I) } \lesssim_{\epsilon} N^{\epsilon  + \frac{n+d}{2} - \frac{n + d + 2}{p}} \bigg( \sum_{m,k} \|e^{it\Delta} g_{\theta_{m,k}}w_{I}\|^{2}_{L^{p}(\R^{n} \times \T^{d} \times \R) } \bigg)^{1/2},$$ where $w_{I}$ is a bump function adapted to $I$.  \end{lemma}

\noindent The proof of Lemma \ref{decLemma} is similar to the proof of the discrete restriction theorem in \cite{BD}. One approximates functions on $\R^{n}\times \T^{d}$ by functions on $\R^{n+d}$, applies the $\ell^{2}$ decoupling theorem, and then takes limits. The details are in Section 3.3 below. 

\begin{lemma} \label{dispLemma} Let $h(x,y,t) = \sum_{l}h_{l}(x,t)e^{2\pi i l y}$ be a Schwartz function on $\R^{n} \times \T^{d} \times [-1,1]$ such that $h_{l}(x,t) \neq 0$ for at most $R$ values of $l$. Let $\psi$ be a bump function on $\R^{n}$ supported in $B_{1}$ and define $$K_{\gamma}(x,y,t) = \sum_{\substack{l \in \Z^{d} \\ |l| \leq cN }} \int_{\R^{n}} \psi (N^{-1}\xi) e^{2\pi i (x\cdot \xi + y\cdot l + (t+\gamma)(|\xi|^{2} + |l|^{2}) )} d\xi$$ for $\gamma \in \Z$ with $|\gamma| \geq 1$. Then for all $p > 2$ \begin{equation} \label{localDisp} \|K_{\gamma} \ast h \|_{L^{p}(\R^{n} \times \T^{d} \times [-1,1])} \lesssim R^{\frac{p-2}{p}}|\gamma|^{-\mu}\|h\|_{L^{p'}(\R^{n} \times \T^{d} \times [-1,1])}, \end{equation} where $\mu = \frac{n(p - 2)}{ 2p}.$ In particular, if $p = p^{\ast}$ we have $\mu = \frac{n}{n+d+2}$. \end{lemma}

\begin{proof} Let $S\subset \Z^{d}$ be the set of $l$ such that $h_{l} \neq 0$. By examining the Fourier coefficients we see that $K_{\gamma} \ast h = K_{\gamma,S}\ast h,$ where $$K_{\gamma,S} =  \sum_{l \in S} \int_{\R^{n}} \psi (N^{-1}\xi) e^{2\pi i (x\cdot\xi + y\cdot l + (t+\gamma)(|\xi|^{2} + |l|^{2}) )} d\xi.$$ By a stationary phase argument and the triangle inequality we have $\|K_{\gamma,S}\|_{L^{\infty}} \lesssim R|\gamma|^{-n/2},$ and therefore $\|K_{\gamma}\ast h\|_{L^{\infty}} \lesssim R|\gamma|^{-n/2}\|h\|_{L^{1}}.$ On the other hand, since $t$ is in a bounded interval we can use Plancharel's theorem to get $\|K_{\gamma,S}\ast h\|_{L^{2}} \lesssim \|h\|_{L^{2}}.$ Interpolating these estimates yields $\|K_{\gamma,S}\ast h\|_{L^{p}} \lesssim R^{\frac{(p-2)}{p}}|\gamma|^{-\mu}\|h\|_{L^{p'}},$ and this completes the proof since $K_{\gamma} \ast h = K_{\gamma,S}\ast h.$ 
	
\end{proof}

\noindent Finally, we record the following local Strichartz estimate. 

\begin{prop}\label{localStrichartz} For any bounded time interval $I$ and $p \geq p^{\ast}$ one has $$\| e^{it\Delta_{\R^{n} \times \T^{d}} }P_{\leq N}f \|_{L^{p} (\R^{n} \times \T^{d} \times I) } \lesssim_{I, \epsilon} N^{\epsilon + \frac{n+d}{2} - \frac{n + d + 2}{p}}\|f\|_{L^{2}(\R^{n} \times \T^{d})}$$ and $$\big\| \int_{I} e^{it\Delta} P_{\leq N}h(x,y,t) dt \big\|_{L^{2}_{x,y}(\R^{n} \times \T^{d})} \lesssim_{I,\epsilon} N^{\epsilon + \frac{n+d}{2} - \frac{n + d + 2}{p}} \|h\|_{L_{x,y,t}^{p'} (\R^{n} \times \T^{d} \times I ) }.$$  \end{prop}

\begin{proof} We prove the first estimate since the second follows easily by duality.
	
	Suppose $f = P_{\leq N}f$. By interpolating with $p = \infty$ (via Bernstein's inequality) it suffices to prove the endpoint case $p = p^{\ast}$. By Lemma \ref{decLemma}
	$$\|e^{it\Delta}f\|_{L^{p}(\R^{n} \times \T^{d} \times I) } \lesssim_{\epsilon,I} N^{\epsilon} \bigg( \sum_{m,k} \|e^{it\Delta}f_{\theta_{m,k}} \cdot w_{I}\|^{2}_{L^{p}(\R^{n}\times \T^{d} \times \R ) } \bigg)^{1/2}.$$ By Plancharel's theorem it suffices to prove the desired estimate when $f = P_{\theta}f$ and $\theta = \theta_{m,k}$. In this case we apply H\"{o}lder's inequality in time to get $$\|e^{it\Delta}f_{\theta} \cdot w_{I}\|_{L^{p}(\R^{n}\times \T^{d} \times \R ) } \lesssim \|e^{it\Delta}f_{\theta}\|_{L_{t}^{q} L_{x}^{p} },$$ where $q = \frac{2(n+d+2)}{n}$ is the admissible time exponent for the $L_{t}^{q}L_{x}^{p}$ Strichartz estimate on $\R^{n}$. Applying this Strichartz bound completes the proof. 
	
\end{proof}

\begin{remark} In the special case $n = d = 1$ we can instead apply the stronger local Strichartz estimate $$\| e^{it\Delta_{\R \times \T} }f \|_{L^{4} (\R \times \T \times I) } \leq C_{I}\|f\|_{L^{2}(\R \times \T)} $$ due to Takaoka and Tzvetkov \cite{TT}. This will simplify the $\epsilon$-removal argument in Section 4 when $n = d = 1$. \end{remark}

\subsection{Main Argument} We prove Proposition \ref{mainThmEps} in the case $p = p^{\ast}$. The the case $p > p^{\ast}$ is essentially the same (or can be obtained by interpolation). Let $q = \frac{4p^{\ast}}{n(p^{\ast} - 2)} = \frac{2(n+d+2)}{n}$ and let $w_{\gamma}$ be a bump function adapted to $[\gamma - 1, \gamma +1]$. We wish to show that \begin{equation}\label{mixedEstPN}
\bigg(\sum_{\gamma \in \Z} \| e^{it\Delta}P_{\leq N}f \cdot w_{\gamma} \|^{q}_{L^{p^{\ast}} (\R^{n} \times \T^{d} \times \R) }\bigg)^{1/q}  \lesssim_{\epsilon} N^{\epsilon}\|f\|_{L^{2}(\R^{n} \times \T^{d})}.\end{equation} We will assume throughout the section that $f = P_{\leq N}f$. Let $M(f)$ denote the mixed norm on the left-hand side of \eqref{mixedEstPN}. In order to show that $M(f) \lesssim_{\epsilon} N^{\epsilon}\|f\|_{L^{2}}$, we decouple the frequencies to reduce to the case where $f$ has small Fourier support. Let $f_{\theta_{m,k}}$ be defined as in Lemma \ref{decLemma}.  Then by applying Lemma \ref{decLemma} for each $\gamma$ and using Minkowski's inequality in $\ell^{ \frac{q}{2}}$ we obtain \begin{align} \nonumber M(f) &\lesssim_{\epsilon} N^{\epsilon}\bigg( \sum_{\gamma} \bigg( \sum_{m,k}\|e^{it\Delta}f_{\theta_{m,k}} \cdot w_{\gamma}\|^{2}_{L^{p^{\ast}}(\R^{n} \times \T^{d}\times \R)} \bigg)^{q/2}   \bigg)^{1/q} \\ \nonumber &\lesssim_{\epsilon}N^{\epsilon} \bigg( \sum_{m,k} \bigg( \sum_{\gamma} \|e^{it\Delta} f_{\theta_{m,k}}\cdot w_{\gamma}\|_{L^{p^{\ast}}(\R^{n} \times \T^{d} \times \R) }^{q}  \bigg)^{2/q}  \bigg)^{1/2} \\ \label{decReduction} &= C_{\epsilon}N^{\epsilon} \bigg(\sum_{m,k}M(f_{\theta_{m,k}})^{2}\bigg)^{1/2}. \end{align} To complete the proof, we claim that \begin{equation} \label{localizedMixed}M(f_{\theta_{m,k}}) \lesssim \|f_{\theta_{m,k}}\|_{L^{2}(\R^{n}\times \T^{d})}\end{equation} for each $m,k$. This will be enough, since by \eqref{decReduction} and Plancharel's theorem we then have $$M(f) \lesssim_{\epsilon} N^{\epsilon} \bigg(\sum_{m,k}M(f_{\theta_{m,k}})^{2}\bigg)^{1/2} \lesssim_{\epsilon} N^{\epsilon} \bigg(\sum_{m,k} \|f_{\theta_{m,k}}\|_{L^{2}}^{2}\bigg)^{1/2} \lesssim_{\epsilon}N^{\epsilon}\|f\|_{L^{2}}. $$ Now \eqref{localizedMixed} shows that in order to prove \eqref{mixedEstPN} we can in fact assume that the Fourier transform of $f$ is supported in a cube of side length $\sim 1$ in the frequency space $\R^{n} \times \Z^{d}$. 

Let $P_{\theta}f$ denote a smooth frequency cut-off of $f$ onto $\theta$. Then by H\"{o}lder's inequality in time one has $$  \bigg(\sum_{\gamma \in \Z} \| e^{it\Delta}P_{\theta}f \cdot w_{\gamma} \|^{q}_{L^{p^{\ast}} (\R^{n} \times \T^{d} \times \R) }\bigg)^{1/q}  \lesssim \bigg(\sum_{\gamma \in \Z} \| e^{it\Delta}P_{\theta}f \cdot w_{\gamma} \|^{q}_{L_{t}^q L^{p^{\ast}}(\R^{n} \times \T^{d}) }\bigg)^{1/q},$$ where $q = \frac{2(n+d+2)}{n} > p^{\ast}$. It is straightforward to check that $\sum_{\gamma}w_{\gamma} \lesssim 1$, and therefore $$\bigg(\sum_{\gamma \in \Z} \| e^{it\Delta}P_{\theta}f \cdot w_{\gamma} \|^{q}_{L_{t}^q L^{p^{\ast}}(\R^{n} \times \T^{d}) }\bigg)^{1/q} \lesssim \| e^{it\Delta}P_{\theta}f  \|_{L_{t}^q L_{x,y}^{p^{\ast}} }.$$ But $P_{\theta}f$ only has one non-zero Fourier coefficient, and therefore we can use the Euclidean $L_{t}^{q}L_{x}^{p^{\ast}}$ Strichartz estimate to finally obtain $$\| e^{it\Delta}P_{\theta}f \|_{L_{t}^q L_{x,y}^{p^{\ast}} } \lesssim \|P_{\theta}f \|_{L^{2}(\R^{n})} \lesssim \|P_{\theta}f \|_{L^{2}(\R^{n} \times \T^{d})},$$ as desired.

 \subsection{Proof of Decoupling Lemma \ref{decLemma}}

Recall $\alpha_{p} = \frac{n+d}{2} - \frac{n + d + 2}{p}.$ Let $B_{l} \subset \R^{n}$ be a fixed ball of radius $N$, and let $w_{l}$ be a smooth weight adapted to $B_{l} \times [-1,1]$. To prove the lemma it will suffice to show that \begin{equation} \label{eq:localDec}
\|e^{it\Delta}g\|_{L^{p}(B_{l} \times \T^{d} \times [-1,1]) } \lesssim_{\epsilon} N^{\epsilon  + \alpha_p } \bigg( \sum_{m,k} \|e^{it\Delta}g_{\theta_{m,k}}\|^{2}_{L^{p}(w_{l}) } \bigg)^{1/2}.\end{equation} Then to prove the full estimate on $\R^{n} \times \T^{d}$, we can choose a finitely-overlapping collection of balls $B_{l}$ of radius $N$ that cover $\R^{n}$, and then apply \eqref{eq:localDec} in each $B_{l}$ and use Minkowski's inequality to sum: \begin{align*}\|e^{it\Delta}g\|^{p}_{L^{p}(\R^{n} \times \T^{d} \times [-1,1]) }  &\leq \sum_{l} \|e^{it\Delta}g\|^{p}_{L^{p}(B_l \times \T^{d} \times [-1,1]) } \\ &\lesssim_{\epsilon} N^{\epsilon + \alpha_p} \sum_{l}\bigg( \sum_{m,k} \|e^{it\Delta}g_{\theta_{m,k}}\|^{2}_{L^{p}(w_{l}) } \bigg)^{p/2}\\ &\lesssim_{\epsilon} N^{\epsilon + \alpha_p} \bigg( \sum_{m,k} \|e^{it\Delta}g_{\theta_{m,k}}\|^{2}_{L^{p}(w)} \bigg)^{p/2}, \end{align*} where the $w_{l}$ are weights adapted to $B_{l} \times [-1,1]$ and $w = \sum_{l}w_{l}.$ Now if $v$ is a smooth weight on $\R$ that rapidly decays outside $[-1,1]$ it is easy to check that $w \leq Cv$, so Lemma \ref{decLemma} will follow.  

Let $u(x,y,t) = e^{it\Delta}g (x,y).$ We begin by rescaling $u_0$ to have frequency support in $[-1,1]^{n+d}$. Note that $$u(N^{-1}x, N^{-1}y, N^{-2}t) = N^{n} \int_{B_{1}^{n}} \sum_{m \in N^{-1}\Z^{d} \cap B^{d}_{1}} \widehat{g}(N\xi, Nm) e^{i(x \cdot \xi + y\cdot m + t(|\xi|^{2} + |m|^{2}) )} \ d\xi,$$ 
where $B^{n}_{1}$ and $B^{d}_{1}$ are $[-1,1]^{n}$ and $[-1,1]^{d}$, respectively. We let $f(\xi, m) = \widehat{g}(N\xi, Nm)$, so if $\Lambda_N = N^{-1}\Z^{d} \cap [-1,1]^{d}$ then $f$ is supported on $B^{n}_1 \times \Lambda_{N}$. Below we will exploit the fact that $f$ can be viewed as a function on $n$-dimensional cubes of size $\sim 1$ that are $\sim N^{-1}$-separated in $B_1^{n+d}$ (this follows from the support property of $\widehat{g}$; we have one cube for each $m$).   Let $Ef$ denote the extension operator $$Ef = \sum_{m \in \Lambda_N} \int_{B^{n}_1}f(\xi, m) e^{i(x \cdot \xi + y\cdot m + t(|\xi|^{2} + |m|^{2}) )} \ d\xi.$$ After applying a change of variables on the spatial side and using periodicity in the $y$ variable, we see that \begin{align*} \|u\|_{L^{p}(B_{N} \times \T^{d} \times [0,1])} &= N^{n}N^{-(n + d + 2)/p} \|Ef\|_{L^{p}(B_{N^2} \times N\T^{d} \times [0,N^2])} \\ \nonumber &= N^{n(1 - \frac{1}{p})-\frac{d + 2}{p}}N^{-\frac{d}{p}}\|Ef\|_{L^{p}(B_{N^2} \times N^2\T^{d} \times [0,N^2])}.\end{align*} In particular, for $p^{\ast} = \frac{2(n+d+2)}{n+d}$ we have \begin{equation}  \label{extReduction}\|u\|_{L^{p^{\ast}}(B_{N} \times \T^{d} \times [0,1])} = N^{\frac{n-d}{2} -\frac{d}{p^{\ast}} }\|Ef\|_{L^{p^{\ast}}(B_{N^2} \times N^2\T^{d} \times [0,N^2])}\end{equation} for any solution $u$ with initial data $u_0$ and $f = \widehat{u_0}(N\xi, Nm).$

We also introduce the operator $\widetilde{E}$ on $C^{\infty}([-1,1]^{n+d})$ defined by  $$\widetilde{E}h(x,y,t) = \int_{B^{d}_1} \int_{B^{n}_1}h(\xi_1, \xi_2) e^{i(x \cdot \xi_1 + y\cdot \xi_2 + t(|\xi_1|^{2} + |\xi_2|^{2}) )} \ d\xi_1 d\xi_2$$ (this is the usual extension operator associated to the paraboloid $P^{n+d}$ in $\R^{n+d+1}$). Given a function $f$ on $[-1,1]^{n}\times \Lambda_N$, let $$f^{\delta}(\xi_1, \xi_2) = \sum_{m \in \Lambda_N}c_d\delta^{-d}1_{\{|\xi_2 - m| \leq \delta\} }f(\xi_1, m), \ \ \ \ \ \delta < \frac{1}{N}$$ where $c_{d}$ is a dimensional constant chosen for normalization. Then by Lebesgue differentiation and Fatou's lemma we have \begin{equation} \label{fatou}
\|Ef\|_{L^{p}(B_{N^2} \times N^2\T^{d} \times [0,N^2])} \leq \liminf_{\delta \rightarrow 0}\|\widetilde{E}f^{\delta}\|_{L^{p}(B_{N^2} \times [0,N^{2}]^{d} \times [0,N^2])}, \end{equation} where as usual we identify $N^{2}\T^{d}$ with $[0,N^{2}]^{d}$. We will begin by estimating $\widetilde{E}f$ for arbitrary $f$ on $B^{n+d}_1$ before specializing to $f^{\delta}$ and passing to the limit later in the argument. 

Fix a parameter $R \lesssim N$ and choose a finitely-overlapping collection of $(n+d)$-dimensional boxes $\theta_{m,k} = I_{k} \times I_{m} \subset B^{n}_1 \times B^{d}_1$ with the following properties: 
\begin{enumerate}[label=(\roman*)] \item Each $\theta_{m,k}$ has side lengths $\sim R^{-1}$  \item $I_{m}$ is centered at $m \in \Lambda_{N}$ \item $m$ varies over an $\sim R^{-1}$-separated subset of $\Lambda_N$ such that the cubes $I_{m}$ cover $[-1,1]^{d}.$ \end{enumerate}

\noindent Such a collection of $\theta_{m,k}$ yields a finitely-overlapping tiling of $B^{n+d}_1$ by boxes of side length $\sim R^{-1}$, which corresponds to a covering of the paraboloid $P^{n+ d}$ by $R^{-1}$-caps. Let $\{\varphi_{m,k} \}$ be a smooth partition of unity subordinate to this cover, and let $f_{m,k} = f \cdot \varphi_{m,k}$. By Theorem \ref{thm:BDparabola} $$\|\widetilde{E}f\|_{L^{p} (B_{N^{2}} \times [0,N^2]^{d} \times [0,N^2])} \lesssim_{\epsilon}R^{\epsilon + \alpha_{p}} \bigg( \sum_{m,k} \|\widetilde{E}f_{m,k}\|^{2}_{L^{p}(w_{N^2})} \bigg)^{\frac{1}{2}}, $$ where $w_{N^2}$ is a bump function adapted to $B_{N^{2}} \times [0,N^2]^{d} \times [0,N^2].$ Setting $R = N$ and $p = p^{\ast}$ yields 

\begin{equation} \label{decouple1} \|\widetilde{E}f\|_{L^{p^{\ast}} (B_{N^2} \times [0,N^2]^{d} \times [0,N^2])} \lesssim_{\epsilon}N^{\epsilon } \bigg( \sum_{m,k} \|\widetilde{E}f_{m,k}\|^{2}_{L^{p^{\ast}}(w_{N^2})} \bigg)^{\frac{1}{2}}, \end{equation} 

\noindent with the $f_{m,k}$ supported on finitely-overlapping boxes of side lengths $\sim \frac{1}{N}$. Specializing to $f^{\delta}$ and taking a limit as in \eqref{fatou}, we get $$\|Ef\|_{L^{p^{\ast}}(B_{N^2} \times N^{2}\T^{d} \times [0,N^2] )} \lesssim_{\epsilon} \liminf_{\delta \rightarrow 0}N^{\epsilon} \bigg( \sum_{m,k} \|\widetilde{E}f^{\delta}_{m,k}\|^{2}_{L^{p^{\ast}}(w_{N^2})} \bigg)^{\frac{1}{2}}.$$ Then as a consequence of the scaling \eqref{extReduction} 

\begin{equation} \label{decFinalStep}\|u\|_{L^{p^{\ast}}(B_N \times \T^{d} \times [0,1])} \lesssim_{\epsilon}\liminf_{\delta \rightarrow 0} N^{\frac{n-d}{2} -\frac{d}{p^{\ast}} + \epsilon }\bigg( \sum_{m,k} \|\widetilde{E}f^{\delta}_{m,k}\|^{2}_{L^{p^{\ast}}(w_{N^2})} \bigg)^{\frac{1}{2}},\end{equation}

\noindent where as above $f(\xi, n) = \widehat{u_0}(N\xi, Nn)$ for $n \in \Lambda_N$. Now by rescaling as before and writing out the definition of $f^{\delta}$ we get $$  N^{\frac{n-d}{2} -\frac{d}{p^{\ast}} }\|\widetilde{E}f^{\delta}_{m,k}\|_{L^{p^{\ast}}(w_{N^2})} \approx  N^{-\frac{d}{p^{\ast}} }\|u^{\delta} \|_{L^{p^{\ast}}(v)}, $$ where $$u^{\delta} = \sum_{\substack{l \in \Z^{d} \\ |l - m| \lesssim 1} } \int_{\R^{d}}\int_{\R^{n}} c(N\delta)^{-d}1_{\{|\xi_2 - l| \leq N\delta\} }f_{m,k}(N^{-1}\xi_1, N^{-1}l) e^{2\pi i (x\cdot\xi_1 + y\cdot\xi_2 + t(|\xi_1|^2 + |\xi_2|^2 ) )} d\xi_1 d\xi_2 $$ and $v$ is a suitable bump function adapted to $B_{N} \times [0,N]^{d} \times [-1,1].$ The desired result now follows from the pointwise estimate $$|u^{\delta}|  \leq \sum_{\substack{l \in \Z^{d} \\ |l - m| \lesssim 1} } \big| \int_{\R^{n}}f_{m,k}(N^{-1}\xi_1, N^{-1}l) e^{2\pi i (x\cdot\xi_1 + t|\xi_1|^2  )} d\xi_1\big| ,$$ which is uniform in $\delta$. 

The case $p > p^{\ast}$ follows by a similar argument (or interpolation).

\section{The $\epsilon$-Removal Argument}

In this section we show that the $N^{\epsilon}$ loss from our Strichartz estimates can be removed for $p > p^{\ast}$, where as before $p^{\ast}$ is the endpoint $p^{\ast} = \frac{2(n+d+2)}{n+d}.$ The argument has a local and global component. In the local case the philosophy is the same as in \cite{B} and \cite{KV}: one applies the Strichartz estimate with $\epsilon$ loss in the region where the operator is in some sense `small,' and this leaves a `large' portion of the operator which we can control with direct estimates for the associated kernel $K$, at least for $p$ away from the endpoint. To extend the $\epsilon$-removal to the global case we combine the local estimates with an interpolation argument to handle the global portion of a relevant bilinear form. 

\subsection{The Local-In-Time Case}

We begin the argument by first showing that one can remove the $N^{\epsilon}$ loss from the local-in-time estimate from Proposition \ref{localStrichartz}.

\begin{prop} \label{prop:localEpsRemoval} Suppose $p > \frac{2(n+d+2)}{n+d}$. Then $$\|e^{it \Delta_{\R^{n}\times \T^{d} }} P_{\leq N} f \|_{L^{p}(\R^{n} \times \T^{d} \times [0,1] )} \lesssim N^{\frac{n+d}{2} - \frac{n + d + 2}{p}}\|f\|_{L^{2}(\R^{n} \times \T^{d} )}.$$
\end{prop}

\noindent To prove the proposition we reduce the problem to a situation where we can apply ideas due to Killip and Visan. The argument in this local setting is almost the same as their proof of $\epsilon$-removal in the case $n = 0$ from \cite{KV}. 

Normalize $\|f\|_{L^{2}} = 1$ and let $F = e^{it \Delta_{\R^{n}\times \T^{d} }} P_{\leq N} f$. Using the level set characterization of the $L^{p}$ norm and Bernstein's inequality we get 
$$\|F \|^{p}_{L^{p}(\R^{n} \times \T^{d} \times [0,1] )} = p\int_{0}^{CN^{ \frac{n+d}{2}}} \mu^{p-1} |\{ (x,y,t) : |F(x,y,t)| > \mu  \}| \ d\mu.$$ Now if $\delta > 0$ then we can apply Chebyshev's inequality and Proposition \ref{localStrichartz} to see that 

\begin{align*}\int_{0}^{N^{\frac{n+d}{2} - \delta}} \mu^{p-1} |\{ F > \mu  \}| \ d\mu &\lesssim_{\epsilon} N^{p^{\ast}(\frac{n+d}{2} -\frac{n + d + 2}{p^{\ast}})} N^{(p - p^{\ast}) (\frac{n+d}{2} - \delta ) + \epsilon} \\ &\leq C_{\epsilon} N^{p(\frac{n+d}{2} -\frac{n + d + 2}{p})}N^{\epsilon - \delta(p - p^{\ast})} \\ &\leq CN^{p(\frac{n+d}{2} -\frac{n + d + 2}{p})} \end{align*} provided $\epsilon$ is chosen small enough. It remains to estimate the large portion of the integral \begin{equation}\label{removalLargePart}A := \int_{N^{\frac{n+d}{2} - \delta}}^{CN^{\frac{n+d}{2}}} \mu^{p-1} |\{ F > \mu  \}| \ d\mu.\end{equation} We let $\Omega = \{ F > \mu  \}$ for $\mu \geq N^{\frac{n+d}{2} - \delta}$ fixed, and set $\Omega_{\omega} = \{ \text{Re}(e^{i\omega}F) > \mu/2 \}$. Note that there is some choice of $\omega \in \{ 0, \frac{\pi}{2}, \pi, \frac{3\pi}{2} \}$ such that $|\Omega| \leq 4 |\Omega_{\omega}|$, so we will estimate $|\Omega_{\omega}|$ instead. In particular 
\begin{equation}\label{localRemovalSet} \mu^{2}|\Omega_{\omega}|^{2} \lesssim \langle \mathbbm{1}_{\Omega_{\omega}}, K \ast \mathbbm{1}_{\Omega_{\omega}} \rangle_{L^{2}_{x,y,t}},  \end{equation}   

\noindent where $K$ is the kernel $$K(x,y,t) = \big(\sum_{l \in \Z^{d}}\psi(N^{-1}l) e^{2\pi i (y\cdot l + t|l|^{2})} \big) \cdot \big( \int_{\R^{n}} \psi (N^{-1}\xi) e^{2\pi i (x\cdot \xi + t|\xi|^{2})} d\xi\big) $$ from above. 

Since $t \in [0,1]$ we cannot take advantage of the dispersion coming from the $\R^{n}$ component of $K$. Instead we discretize the kernel and proceed as if we were working on $\T^{n+d}$. 

\begin{lemma}\label{lem:discrete} Let $K_{\R}(x,t) = \int_{\R^{n}} \psi (N^{-1}\xi) e^{2\pi i (x\cdot \xi + t|\xi|^{2})} d\xi.$ Then if $Q_0$ is the cube of side-length 1 centered at the origin we have $$|K_{\R}(x,t)| \lesssim  \sup_{\alpha \in Q_{0}} \big| \sum_{\substack{ k \in \Z^{n} \\ |k| \leq cN  } } \psi(N^{-1}(\alpha + k)) e^{2\pi i [ (x + 2t\alpha)\cdot k + t|k|^{2} ] }  \big|$$
\end{lemma}

\begin{proof}
	Let  $Q_k = Q_0 + k$, where $k \in \Z^{n}$. We can write$$K_{\R}(x,t) = \sum_{k \in \Z^{n}} \int_{Q_k} \psi (N^{-1}\xi) e^{2\pi i (x\cdot \xi + t|\xi|^{2})} d\xi.$$ Now making the change of variables $\xi = \alpha + k$ for each $k$, where $\alpha \in Q_0$, we get $$K_{\R}(x,t) = \int_{Q_0} e^{2\pi i (x \cdot \alpha + t|\alpha|^{2})} \cdot  \sum_{\substack{ k \in \Z^{n} \\ |k| \leq cN }} \psi(N^{-1}(\alpha + k)) e^{2\pi i [ (x + 2t\alpha)\cdot k + t|k|^{2} ] } \ d\alpha.$$ The lemma follows by taking absolute values inside the integral and taking the supremum in $\alpha$. 
\end{proof}

We will now use a pointwise estimate for $K$ originally due to Bourgain \cite{B}. Given integers $1 \leq q < N$ and $1 \leq a < q$ with $(a,q) = 1$, let $S_{q,a} = \{ t \in [0,1] : |t - \frac{a}{q}| \leq \frac{1}{qN} \}$. We apply Lemma \ref{lem:discrete} and the proof of Lemma 3.18 in \cite{B} for fixed $\alpha$ to obtain \footnote{ Recall that Dirichlet's approximation theorem implies that either $t \in S_{q,a}$ for some pair $(q,a)$, or $t$ is close to 0.} 
\begin{equation}\label{WeylEst} |K(x,y,t)| \mathbbm{1}_{S_{q,a}}(t) \lesssim \frac{N^{n+d}}{q^{\frac{n+d}{2}} (1 + N^{n+d}|t - \frac{a}{q}|^{\frac{n+d}{2}}) }. \end{equation} 

\noindent Following Killip and Visan \cite{KV}, define the `large' set $$\mathcal{T} = \big\{t \in [0,1] : qN^{2}|t - a/q| \leq N^{2\rho} \text{ for some } q \leq N^{2\rho}, \text{ and } (a,q) = 1 \big\},$$ where $\rho > 0$ is a small parameter to be determined below. Also define $$\widetilde{K}(x,y,t) = K(x,y,t)\mathbbm{1}_{\mathcal{T}}(t),$$ and observe that from \eqref{WeylEst} 
\begin{equation} \label{smallKernelEst} |K - \widetilde{K}| \lesssim N^{n + d - (n+d)\rho}.  \end{equation} This implies that $$|\langle \mathbbm{1}_{\Omega_{\omega}}, (K - \widetilde{K}) \ast \mathbbm{1}_{\Omega_{\omega}} \rangle_{L^{2}_{x,y,t}}| \lesssim N^{n + d - (n+d)\rho}|\Omega_{\omega}|^{2},$$ and since $\mu \geq N^{\frac{n+d}{2} - \delta}$ we can absorb the contribution of $K - \widetilde{K}$ to the left-hand side of \eqref{localRemovalSet} provided we take $\rho < \frac{2\delta}{n+d}.$ To estimate the contribution from $\widetilde{K}$ we use the following result due to Killip and Visan:

\begin{prop}[\cite{KV}] \label{prop:KV}Suppose $r > \frac{2(n+d+2)}{n+d}.$ Then $$\|\widetilde{K} \ast F \|_{L^{r}(\R^{n} \times \T^{d} \times [0,1] )} \lesssim N^{2(\frac{n+d}{2} -\frac{n+d+2}{r} )}\|F\|_{L^{r'}(\R^{n} \times \T^{d} \times [0,1] )},$$ provided $\rho$ is chosen small enough (depending only on $n,d,r$). \end{prop}

\begin{proof} The estimate follows by applying the argument from \cite{KV}, Section 2, after bounding the kernel $K$ pointwise via Lemma \ref{lem:discrete} and \eqref{WeylEst}. \end{proof}

Applying this proposition to \eqref{localRemovalSet}, we obtain $$\mu^{2}|\Omega_{\omega}|^{2} \lesssim |\langle \mathbbm{1}_{\Omega_{\omega}}, \widetilde{K} \ast \mathbbm{1}_{\Omega_{\omega}} \rangle_{L^{2}_{x,y,t}}| \lesssim |\Omega_{\omega}|^{\frac{2}{r'}}N^{2(\frac{n+d}{2} -\frac{n+d+2}{r})}.$$ We can then conclude that \begin{equation} \label{OmegaFianlEst} |\Omega| \leq 4|\Omega_{\omega}| \lesssim N^{\frac{r}{2}(n+d -  \frac{2(n+d+2)}{r})}\mu^{-r}\end{equation} for any $r \in (\frac{2(n+d+2)}{n+d}, p).$ Plugging \eqref{OmegaFianlEst} into \eqref{removalLargePart} gives \begin{align*}A &\lesssim N^{\frac{r}{2}(n+d -  \frac{2(n+d+2)}{r})} \int_{N^{\frac{n+d}{2} - \delta}}^{CN^{\frac{n+d}{2}}} \mu^{p - r - 1} d\mu \\ &\lesssim  N^{p(\frac{n+d}{2} - \frac{n + d + 2}{p})},\end{align*} completing the proof of Proposition \ref{prop:localEpsRemoval}.

\vspace{2mm} 

\begin{remark} Proposition \ref{prop:KV} remains true if $\T^{d}$ is replaced by a $d$-dimensional irrational torus, after a suitable modification of the pointwise estimate \eqref{WeylEst} to take into account the irrational parameters. See \cite{KV} for more details.\end{remark}

\subsection{The Global Case} We now extend Proposition \ref{prop:localEpsRemoval} to the global setting and show that for all $p > p^{\ast}$ and $q = q(p) = \frac{4p}{n(p-2)}$,  $$ \bigg(\sum_{\gamma \in \Z} \| e^{it\Delta}P_{\leq N}f \|^{q}_{L^{p} (\R^{n} \times \T^{d} \times [\gamma -1, \gamma +1]) }\bigg)^{1/q}  \lesssim N^{\frac{n+d}{2} - \frac{n + d + 2}{p}}\|f\|_{L^{2}(\R^{n} \times \T^{d})}.$$ This will complete the proof of Theorem \ref{mainThm}. 

The argument is an adaptation of the bilinear interpolation approach of Keel and Tao \cite{KT}, and is also inspired by the $TT^{\ast}$ argument of Hani and Pausader \cite{HP}. We will use the fact that we already have both the local estimate without $\epsilon$-loss, and the global estimate for $q = q(p)$ with $\epsilon$ loss. We assume below that $p > p^{\ast}  = \frac{2(n+d+2)}{n+d}$ is fixed. We are free to take $p$ as close to $p^{\ast}$ as needed, since the remaining cases can be handled by interpolation. 

\subsubsection{Initial Reductions} Let $U_{\alpha}(t) = e^{-i(t+\alpha)\Delta_{\R^{n}\times \T^{d}}}$, and for functions $h(x,y,t)$ let $h_{\alpha}(x,y,t) = h(x,y, t + \alpha)$. By appealing to duality, we see that if $T(h,g)$ is the bilinear form \begin{equation}\label{eq:bilinear} T(h,g) = \sum_{\alpha \in \Z}\sum_{\gamma \in \Z} \int_{0}^{1}\int_{0}^{1} \langle U_{\alpha}(s)P_{\leq N}h_{\alpha}(s), U_{\alpha+\gamma}(t)P_{\leq N} g_{\alpha+\gamma}(t) \rangle_{L_{x,y}^{2}} dsdt
\end{equation}

\noindent then it suffices  to show that for $p > \frac{2(n+d+2)}{n+d}$ and $q = q(p) = \frac{4p}{n(p-2)}$ we have \begin{equation}\label{eq:formEst} |T(h,g)| \lesssim N^{2\alpha_{p}} \|h\|_{\ell_{\gamma}^{q'} L^{p'}(\R^{n} \times \T^{d}\times [\gamma, \gamma +1]) } \|g\|_{\ell_{\gamma}^{q'} L^{p'}(\R^{n} \times \T^{d}\times [\gamma, \gamma +1])}.
\end{equation}

\noindent  Note that we can immediately prove \eqref{eq:formEst} for the diagonal portion of $T$ where $|\gamma|\leq 10$ by using Cauchy-Schwarz and Proposition \ref{prop:localEpsRemoval} (in fact this argument gives a stronger estimate with an $\ell^{2}$ sum). Hence we can assume in \eqref{eq:bilinear} that $|\gamma| \geq 10$. We can also assume that $h = P_{\leq N}h$ and $g = P_{\leq N} g$.

We dyadically decompose $T$ and for $j \geq 3$ define $$T_{j}(h,g) = \sum_{\alpha \in \Z} \sum_{\substack{\gamma \in \Z \\ 2^{j} \leq |\gamma| < 2^{j+1} } }  \int_{0}^{1}\int_{0}^{1} \langle U_{\alpha}(s)h_{\alpha}(s), U_{\alpha+\gamma}(t)g_{\alpha+\gamma}(t) \rangle_{L_{x,y}^{2}} dsdt,$$ with the goal of showing that $$\sum_{j \geq 3} |T_{j}(h,g)| \lesssim  N^{2\alpha_{p}} \|h\|_{\ell_{\gamma}^{q'} L^{p'}(\R^{n} \times \T^{d}\times [\gamma, \gamma +1]) } \|g\|_{\ell_{\gamma}^{q'} L^{p'}(\R^{n} \times \T^{d}\times [\gamma, \gamma +1])}.$$

We claim that we can assume $h_{\alpha}(s)$ is zero for $\alpha$ outside an interval of length $2^{j}$. Indeed, let $I_{l} = \{\alpha: l2^{j} \leq \alpha < (l+1)2^{j} \} $ and suppose that for some pair of exponents $(a,b)$ we have $$\sup_{l \in \Z} |T_{j}(h\mathbbm{1}_{I_{l}} , g )| \leq A \|h\|_{\ell^{q'} L^{a'}} \|g\|_{\ell^{q'} L^{b'}}.$$ Note that for fixed $l$ the terms in $T_{j}(h,g)$ are zero unless $\alpha+ \gamma \in 2I_{l+1}$. Then  
\begin{align*}|T_{j}(h,g)| &\leq A\sum_{l} \|h \mathbbm{1}_{I_{l}}\|_{\ell^{q'}L^{a'} }  \|g \mathbbm{1}_{2I_{l+1}}\|_{\ell^{q'}L^{b'} } \\ & \leq A \big( \sum_{l} \|h \mathbbm{1}_{I_{l}}\|_{\ell^{q'}L^{a'} }^{q'} \big)^{\frac{1}{q'}} \big( \sum_{l} \|g \mathbbm{1}_{2I_{l+1}}\|_{\ell^{q'}L^{b'} }^{q} \big)^{\frac{1}{q}} \\ &\leq A \big( \sum_{l} \|h \mathbbm{1}_{I_{l}}\|_{\ell^{q'}L^{a'} }^{q'} \big)^{\frac{1}{q'}} \big( \sum_{l} \|g \mathbbm{1}_{2I_{l+1}}\|_{\ell^{q'}L^{b'} }^{q'} \big)^{\frac{1}{q'}} \\ &\leq cA \|h\|_{\ell^{q'} L^{a'}} \|g\|_{\ell^{q'} L^{b'}}, \end{align*} using the fact that $q > q'$ in the second-to-last inequality.   Hence it suffices to estimate $T_{j}(h,g)$ when $h_{\alpha}$ is supported with $\alpha \in I_{l}$ for some $l$, which additionally implies that we can assume $g$ is time-supported in an interval of $J$ length $\sim 2^{j}$. Note that with these assumptions $T_{j}(h,g) = T(h \mathbbm{1}_{I_{l}}, g \mathbbm{1}_{J} ),$ so we will be able to take advantage of the dual of Proposition \ref{mainThmEps}. 

\subsubsection{The First Interpolation} The first step is to prove the following two-parameter family of estimates. The result is similar to Lemma 4.1 in Keel-Tao \cite{KT}, though we have to interpolate in a smaller range to avoid too large of a loss in $N$.

\begin{lemma} \label{lemma:interpolation1} For all $(\frac{1}{a}, \frac{1}{b})$ in a neighborhood of $(\frac{1}{p}, \frac{1}{p})$ with $a,b > p^{\ast}$ we have $$ |T_{j}(h,g)| \lesssim_{\epsilon} N^{c(a,b) + \epsilon} 2^{j\beta(a,b)}\|h\|_{\ell^{q'}L^{a'}}\|g\|_{\ell^{q'}L^{b'}}, $$ where\begin{equation*}
	c(a,b)=\begin{cases}
	(1 - \frac{p^{\ast}}{a})d + \frac{n+d+2}{a} - \frac{n+d+2}{b}, & \text{if $a \leq b$}\\
	(1 - \frac{p^{\ast}}{b})d + \frac{n+d+2}{b} - \frac{n+d+2}{a}, & \text{if $a > b$}
	\end{cases}
	\end{equation*}and $$\beta(a,b) = \frac{n}{2a} + \frac{n}{2b} - \frac{n}{p}.$$
\end{lemma} 

\begin{proof} We begin by proving the lemma in the case $(a,b) = (\infty,\infty)$ and then in the two symmetric cases where $a = r$ and $r< b < p$, and where $b = r$ and $r < a < p$.  The full range of estimates is then obtained by interpolating between these cases. Recall from above that we can assume $h$ and $g$ have time support in an interval of length $\sim 2^{j}$. 
	
First consider $(a,b) = (\infty, \infty)$. Since $|\gamma| \geq 10$ in the definition of $T_{j}$ we can use kernel estimates as in Lemma \ref{dispLemma} to get $$|T_{j}(h,g)| \lesssim N^{d} \sum_{\alpha,\gamma} |\gamma|^{-n/2} \|h_{\alpha}\|_{L^{1}} \|g_{\alpha+\gamma}\|_{L^{1}}.$$ Letting $\mu = \frac{n(p-2)}{2p}$ and using the fact that $|\gamma|\sim 2^{j}$, this implies $$|T_{j}(h,g)| \lesssim N^{d} 2^{-j\frac{n}{p}} \sum_{\alpha, \gamma} |\gamma|^{-\mu} \|h_{\alpha}\|_{L^{1}} \|g_{\alpha+\gamma}\|_{L^{1}},$$ and then by the discrete Hardy-Littlewood-Sobolev inequality \begin{equation}\label{interpInfty} |T_{j}(h,g)| \lesssim N^{d}2^{-j\frac{n}{p}}\|h\|_{\ell^{q'}L^{1}}\|g\|_{\ell^{q'}L^{1}}.\end{equation} This proves the lemma when $(a,b) = (\infty,\infty)$.
	
	Next, suppose that $a = p^{\ast}$ and $p^{\ast} < b < p$. By bringing the sums and time integrals into the inner product in $T_{j}$ and then applying Cauchy-Schwarz we get $$|T_{j}(h,g)| \lesssim \sup_{I,I'} \big\| \int_{\R} e^{-it\Delta} (h(t)\mathbbm{1}_{I}) dt  \big\|_{L^{2}_{x,y}}\big\| \int_{\R} e^{-it\Delta} (g(t)\mathbbm{1}_{I'}) dt  \big\|_{L^{2}_{x,y}},$$ where $I,I'$ are time intervals of length $\sim 2^{j}$. Applying the dual form of Theorem \ref{mainThm} and using H\"{o}lder's inequality in time (taking into account the fact that both functions have bounded time support) yields the following: $$|T_{j}(h,g)| \lesssim_{\epsilon} N^{\alpha_{b} + \epsilon} 2^{j ( \frac{1}{q(p^{\ast})'} - \frac{1}{q'} ) } 2^{j ( \frac{1}{q(b)'} - \frac{1}{q'} )} \|h\|_{\ell^{q'}L^{(p^{\ast})'} } \|g\|_{\ell^{q'}L^{b'}}.$$ Here we are using the fact that if $s < p$ then $q(s) > q(p)$, hence $q(s)' < q(p)'$. Now $$\alpha_{b} = \frac{n+d}{2} - \frac{n + d + 2}{b} = \frac{n+d+2}{p^{\ast}} - \frac{n+d+2}{b},$$  so this gives the desired power of $N$. Moreover, we have \begin{align*}\big(\frac{1}{q(p^{\ast})'} - \frac{1}{q'}\big) + \big(\frac{1}{q(b)'} - \frac{1}{q'}\big)  &= \frac{2}{q} - \frac{1}{q(p^{\ast})} - \frac{1}{q(b)}.  \end{align*} It is easy to check using the definition of $q(s)$ that the last expression simplifies to $\beta(p^{\ast},b) = \beta(a,b)$, giving the desired power of $2^{j}$ as well. The proof of the symmetric case $b = p^{\ast}$ and $p^{\ast} < a < p$ is essentially the same. 
\end{proof}

\subsubsection{The Second Interpolation} The fact that we have a two-parameter family of estimates in a neighborhood of $(\frac{1}{p}, \frac{1}{p})$ in Lemma \ref{lemma:interpolation1} will allow us to sum in $j$, but we first have to suitably decompose the input functions $h,g$. 

In particular we will use the atomic decomposition due to Keel and Tao.  

\begin{lemma}[\cite{KT}] Let $(X, \mu)$ be a measure space and $0 < p < \infty$. Then any $f \in L^{p}(X)$ can be written as $f = \sum_{k \in \Z} c_k \chi_{k}$ where each $\chi_{k}$ is a function bounded by $O(2^{-k/p})$ and supported on a set of measure $O(2^{k})$, and the $c_{k}$ are non-negative constants such that $\|c_{k}\|_{\ell^{p}} \lesssim \|f\|_{L^{p}}.$
\end{lemma} 

\noindent We apply the lemma on $L^{p'}(\R^{n} \times \T^{d} \times [\gamma, \gamma +1])$ for each $\gamma$. This allows us to write $$h(x,y,t)\mathbbm{1}_{[\gamma, \gamma+1]}(t) = \sum_{k \in \Z } h^{\gamma}_{k} \chi^{\gamma}_{k}(x,y,t)\mathbbm{1}_{[\gamma, \gamma+1]}(t) $$ with $\chi^{\gamma}_{k}$ supported on a set of measure $O(2^{k})$ and $|\chi^{\gamma}_{k}| \lesssim 2^{-k/p'}$ for each $k$, such that $$\big(\sum_{k}  |h^{\gamma}_{k}|^{p'} \big)^{1/p'} \lesssim \|h\|_{L^{p'}(\R^{n} \times \T^{d} \times [\gamma, \gamma +1])}$$ and hence \begin{equation}\label{eq:hLayer}\big\|\big(\sum_{k}  |h^{\gamma}_{k}|^{p'} \big)^{1/p'} \big\|_{\ell_{\gamma}^{q'}} \lesssim \|h\|_{\ell_{\gamma}^{q'}L^{p'}(\R^{n} \times \T^{d} \times [\gamma, \gamma +1])}.\end{equation} Likewise we can decompose $$g(x,y,t)\mathbbm{1}_{[\gamma, \gamma+1]}(t) = \sum_{m \in \Z } g^{\gamma}_{m} \varphi^{\gamma}_{m}(x,y,t)\mathbbm{1}_{[\gamma, \gamma+1]}(t) $$ with $\varphi^{\gamma}_{m}$ supported on a set of measure $O(2^{m})$ and $|\varphi^{\gamma}_{m}| \lesssim 2^{-m/p'}$ for each $m$, such that \begin{equation}\label{eq:gLayer}\big\|\big(\sum_{m}  |g^{\gamma}_{m}|^{p'} \big)^{1/p'} \big\|_{\ell_{\gamma}^{q'}} \lesssim \|g\|_{\ell_{\gamma}^{q'}L^{p'}(\R^{n} \times \T^{d}\times [\gamma, \gamma +1])}.\end{equation}

Define $h_{k} \chi_{k}$ such that  $h_{k}\chi_{k} = h^{\gamma}_{k} \chi^{\gamma}_{k}$ when $t \in [\gamma , \gamma +1)$, and similarly define $g_{m} \varphi_{m}$. Then we have \begin{equation} \label{decomp1} \sum_{j \geq 3} |T_{j}(h,g)| \leq \sum_{j \geq 3} \sum_{k \in \Z} \sum_{m \in \Z} |T_{j}(h_{k} \chi_{k}, g_{m} \varphi_{m} )|.\end{equation} 

Fix one such pair $(k,m)$. Then for all $(\frac{1}{a}, \frac{1}{b})$ in a neighborhood of $(\frac{1}{p}, \frac{1}{p})$  Lemma \ref{lemma:interpolation1} implies \begin{align}\nonumber |T_{j}(h_{k}\chi_{k}, g_{m}\varphi_{m})| &\lesssim_{\epsilon} 2^{j\beta(a,b)} N^{c(a,b) + \epsilon} \|\chi_{k} h_{k}\|_{\ell_{\gamma}^{q'}L^{a'}} \|\varphi_{m} g_{m}\|_{\ell_{\gamma}^{q'}L^{b'}} \\ \nonumber &\lesssim_{\epsilon} 2^{j\beta(a,b)} N^{c(a,b) + \epsilon} 2^{-\frac{k}{p'}} 2^{-\frac{m}{p'}} 2^{\frac{k}{a'}}2^{\frac{m}{b'}} \|h^{\gamma}_{k}\|_{\ell_{\gamma}^{q'}} \|g^{\gamma}_{m}\|_{\ell_{\gamma}^{q'}} \\ \label{decomp2} &\lesssim_{\epsilon} N^{c(a,b) + \epsilon} 2^{(j-\frac{2}{n}k)( \frac{n}{2a} - \frac{n}{2p})}2^{(j-\frac{2}{n}m)(\frac{n}{2b} - \frac{n}{2p})} \|h^{\gamma}_{k}\|_{\ell_{\gamma}^{q'}} \|g^{\gamma}_{m}\|_{\ell_{\gamma}^{q'}} \end{align} (with implicit constants depending on $\epsilon, a, b, p$). We now optimize in $a$ and $b$ to get \begin{equation}\label{eq:interpK}2^{(j-\frac{2}{n}k)( \frac{n}{2a} - \frac{n}{2p})}  = 2^{-\eta|j-\frac{2}{n}k|}   \end{equation} \begin{equation}\label{eq:interpM}2^{(j-\frac{2}{n}m)( \frac{n}{2b} - \frac{n}{2p})}  = 2^{-\eta|j-\frac{2}{n}m|}  \end{equation} for some uniform $\eta > 0$. We explain how to control the power of $N$ appearing in \eqref{decomp2} in the following lemma.

\begin{lemma}\label{lem:Npower} If $p$ is close enough to $p^{\ast}$ and $(a,b)$ is in a small ball around $(p,p)$ then  \begin{equation}\label{eq:interpNterm} c(a,b) < 2 \alpha_{p}. \end{equation} Moreover, the size of the ball only depends on $n,d$. \end{lemma}

\begin{proof} Suppose without loss of generality that $a \leq b$. Then $c(a,b) < 2 \alpha_{p}$ is equivalent to the inequality \begin{equation}\label{eq:NpowerReduction} (1- \frac{p^{\ast}}{a})d + \frac{n+d+2}{a} - \frac{n+d+2}{b} + \frac{2(n+d+2)}{p} < n + d. \end{equation} Note that $\frac{1}{p} < \frac{1}{p^{\ast}}$, hence $$\frac{2(n+d+2)}{p} < \frac{2(n+d+2)}{p^{\ast}} = n + d .$$ The lemma will then follow if we can choose $a,b$ such that \begin{equation}\label{eq:NpowerReduction2} (1- \frac{p^{\ast}}{a})d + \frac{n+d+2}{a} - \frac{n+d+2}{b} < n+d - \frac{2(n+d+2)}{p}.\end{equation} This is possible for any $a,b$ close enough to $p$, provided $p$ has been chosen sufficiently close to $p^{\ast}$ (which we are always free to assume). In particular, if we set $a = b = p^{\ast}$ then the left-hand side of \eqref{eq:NpowerReduction2} is simply 0. But by continuity the strict inequality $\eqref{eq:NpowerReduction2}$ is preserved if we vary $a,b$ in a small neighborhood around $p$, as long as $p$ is close enough to $p^{\ast}$. This proves $\eqref{eq:NpowerReduction}$ for $(a,b)$ in some small neighborhood of $(p,p)$, and the size of the neighborhood clearly only depends on $p^{\ast}$ and hence only on $n,d$.    \end{proof}

By combining Lemma \ref{lem:Npower} with \eqref{eq:interpK} and \eqref{eq:interpM}, \begin{equation}\label{eq:formEstkm} |T_{j}(h_{k}\chi_{k}, g_{m}\varphi_{m})|  \lesssim N^{2\alpha_{p}} 2^{-\eta|j-\frac{2}{n}k|}2^{-\eta|j-\frac{2}{n}m|}\|h^{\gamma}_{k}\|_{\ell_{\gamma}^{q'}} \|g^{\gamma}_{m}\|_{\ell_{\gamma}^{q'}}. \end{equation} Indeed, by Lemma \ref{lem:Npower} we can choose $\epsilon$ small enough (depending only on $p,n,d$ and the choice of $a,b$) such that $c(a,b) + \epsilon \leq 2\alpha(p).$ 

Summing in $j$ then yields \begin{align*} \sum_{k}\sum_{m}\sum_{j \geq 3} |T_{j}(h_{k}\chi_{k}, g_{m}\varphi_{m})|  &\lesssim N^{2\alpha_{p}} \sum_{k}\sum_{m} \sum_{j \geq 3 }2^{-\eta|j-\frac{2}{n}k|}2^{-\eta|j-\frac{2}{n}m|}\|h^{\gamma}_{k}\|_{\ell_{\gamma}^{q'}} \|g^{\gamma}_{m}\|_{\ell_{\gamma}^{q'}} \\ &\lesssim N^{2\alpha_{p}}\sum_{k} \sum_{m} (1 + \frac{2}{n}|k-m|)2^{-\frac{2}{n}\eta|k-m|} \|h^{\gamma}_{k}\|_{\ell_{\gamma}^{q'}} \|g^{\gamma}_{m}\|_{\ell_{\gamma}^{q'}}. \end{align*} The right-hand side of this bound is of the form $$\sum_{m,k} f(m-k) \big( \sum_{\gamma} |h^{\gamma}_{k}|^{q'} \big)^{1/q'}\big( \sum_{\gamma} |g^{\gamma}_{m}|^{q'} \big)^{1/q'} := \sum_{m,k} f(m-k)c_{k}d_{m}$$ where $f(m)$ is summable in $m$. By H\"{o}lder's and Young's inequality we can control this by $C_{f}\|c\|_{\ell^{p'}} \|d\|_{\ell^{p}}$. Now $p' > q'$ so by Minkowski's inequality $$\|c\|_{\ell^{p'}} = \big( \sum_{k} \big( \sum_{\gamma} |h^{\gamma}_{k}|^{q'} \big)^{\frac{p'}{q'}} \big)^{\frac{1}{p'}} \leq  \big( \sum_{\gamma} \big( \sum_{k} |h^{\gamma}_{k}|^{p'} \big)^{\frac{q'}{p'}} \big)^{\frac{1}{q'}},$$ which we can control by $\|h\|_{\ell^{q'}L^{p'}}$ using \eqref{eq:hLayer}. Also note $\|d\|_{\ell^{p}} \leq \|d\|_{\ell^{p'}}$ since $p > p'$, so the same argument (using \eqref{eq:gLayer}) works to give an acceptable bound for this term as well. We ultimately get $$\sum_{k}\sum_{m}\sum_{j \geq 3} |T_{j}(h_{k}\chi_{k}, g_{m}\varphi_{m})| \lesssim N^{2\alpha_{p}} \|h\|_{\ell^{q'}L^{p'}} \|g\|_{\ell^{q'}L^{p'}},$$ as desired.

\subsection{Proof of Corollary \ref{cor:mixedDeriv}} We now explain how Corollary \ref{cor:mixedDeriv} follows from Theorem \ref{mainThm}, following some observations of Hani and Pausader (see Remark 3.5 in \cite{HP}). The main refinement comes from the following lemma, which is an extension of Corollary 3.4 in \cite{HTT} to our setting.  

\begin{lemma}\label{lem:localRefinement} Suppose $1 \leq M \leq N$ and let $P_{\leq M}^{x}$ denote an $n$-dimensional Littlewood-Paley cut-off to scale $M$ in the $x$ frequency. Then for any $p > \frac{2(n+ d+ 2)}{n+d}$ there is $\delta > 0$ such that $$ \|P_{\leq N} P_{\leq M}^{x} e^{it\Delta_{\R^{n} \times \T^{d} }} u_{0} \|_{ L^{p} (\R^{n} \times \T^{d} \times [-1, 1]) } \lesssim \bigg( \frac{M}{N} \bigg)^{\delta} N^{\frac{n+d}{2} - \frac{n + d + 2}{p}}\|u_{0}\|_{L^{2}(\R^{n} \times \T^{d})}.$$ In particular one can take $0 < \delta < \frac{n}{2}(1- \frac{p^{\ast}}{p})$ with $p^{\ast} = \frac{2(n+d+2)}{n+d}$. 
\end{lemma}

\noindent The lemma follows by interpolating between the estimate \begin{equation} \label{eq:rectEst} \|P_{\leq N} P_{\leq M}^{x} e^{it\Delta_{\R^{n} \times \T^{d} }} u_{0} \|_{ L^{\infty} (\R^{n} \times \T^{d} \times [-1, 1]) } \lesssim M^{\frac{n}{2}}N^{\frac{d}{2}} \|u_{0}\|_{L^{2}(\R^{n} \times \T^{d})}
\end{equation} and the local Strichartz estimate in Proposition \ref{prop:localEpsRemoval}. The estimate \eqref{eq:rectEst} follows from Cauchy-Schwarz and the observation that the Fourier transform of $P_{\leq N} P_{\leq M}^{x}  u_{0}$ is supported in a rectangle that has $\R^n \times \Z^d$ measure $\sim N^{n}M^{d}.$ 

We can now apply Lemma \ref{lem:localRefinement} to obtain a refinement in the estimate of the diagonal portion of the bilinear form \eqref{eq:bilinear}. For the estimate of the non-diagonal portion of the the form \eqref{eq:bilinear} recall that we actually obtained a bound with constant $N^{2\beta}$ and $2\beta < 2\alpha_{p}$. Corollary \ref{cor:mixedDeriv} follows.  

\section{Some Applications}

In this section we show how Theorem \ref{mainThm} can be used to prove Theorem \ref{nonlinearThm}. Below all mixed norms of the type $\ell^{q}L^{p}$ are defined as in Theorem \ref{mainThm}.

\subsection{Function Spaces} We will employ the atomic and variational spaces that have frequently been used to study well-posedness problems for dispersive equations (for examples see \cite{HHK}, \cite{HTT}, \cite{HP}, and \cite{KV}). We recall some basic definitions and properties and refer the reader to \cite{HHK} for proofs. 

Let $I \subset \R$ be a time interval. Given $1 \leq p < \infty$ and a Hilbert space $H$, a $U^{p}(I, H)$ atom $a$ is defined to be a function $a: I \rightarrow H$ such that $$a(t) = \sum_{k=1}^{K} \mathbbm{1}_{[t_{k-1}, t_{k})} \phi_{k-1}, \ \ \ \ \ \ \phi_{k} \in H$$ for some  partition $-\infty < t_0 < t_{1} < ... < t_{K} \leq \infty$, with the additional property that $\sum_{k=0}^{K-1}\|\phi_{k}\|^{p}_{H} = 1$. Then $U^{p}(I, H)$ is the Banach space of functions $u: I \rightarrow H$ with a decomposition of the form $$u = \sum_{j} \lambda_{j} a_{j},$$ where $\{ \lambda_{j} \} \in \ell^{1}(\mathbb{C})$ and $a_{j}$ are $U^{p}(I, H)$ atoms. The norm on $U^{p}$ is taken to be $$\|u\|_{U^{p}} = \inf \{ \sum_{j} |\lambda_{j}| : u = \sum_{j} \lambda_{j} a_{j} \text{ with } a_{j} \ U^{p}\text{-atoms} \}.$$  We also define the variational space $V^{p}(I, H)$ to be the Banach space of functions $v$ such that $$\|v\|_{V^{p}} := \sup_{\text{partitions } \{t_{k}\}} \bigg( \sum_{k=1}^{K} \|v(t_k) - v(t_{k-1})\|^{p}_{H} \bigg)^{\frac{1}{p}} < \infty.$$ We have the important duality relationship $$(U^{p})^{\ast} = V^{p'}.$$ 

We introduce two futher spaces $X^{s}$ and $Y^{s}$ which we will use to carry out the main iteration argument. First let $U_{\Delta}^{p}$ denote the space of functions such that $e^{-it\Delta}u \in U^{p}$, and similarly define $V_{\Delta}^{p}$. Let $\{C_{z}\}_{z \in \Z^{n+d}}$ be a tiling of frequency space by cubes of side-length $\sim 1$, and define $X_{0}^{s}(\R)$ and $Y^{s}(\R)$ to be the spaces of functions $u$ and $v$, respectively, such that the following norms are finite: 
$$ \|u\|_{X_0^{s} (\R) }^{2} := \sum_{z \in \mathbb{Z}^{n+d}} \langle z \rangle^{2s} \|P_{C_{z}}u\|_{ U_{\Delta}^{2} (\R, L^{2}(\R^{n}\times \T^{d} )) }^{2}$$ 
$$ \|v\|_{Y^{s} (\R) }^{2} := \sum_{z \in \mathbb{Z}^{n+d}} \langle z \rangle^{2s} \|P_{C_{z}}v\|_{ V_{\Delta}^{2} (\R, L^{2}(\R^{n}\times \T^{d} )) }^{2}.$$  
\noindent One can likewise define the time-restriction norms $X_{0}^{s}(I), Y^{s}(I)$ for $I \subset \R$ (see \cite{HP}). Recall that we have the sequence of embeddings $$ U_{\Delta}^{2}(\R, H^{s}) \hookrightarrow X_{0}^{s}(\R) \hookrightarrow Y^{s}(\R) \hookrightarrow V_{\Delta}^{2}(\R, H^{s}) \hookrightarrow U_{\Delta}^{p} (\R, H^{s}) \hookrightarrow L^{\infty}(\R, H^{s}) $$ for any $p > 2$. Finally, as in \cite{HP} we introduce the modified space $X^{s}$ defined to be $$X^{s}(\R) := \{u: \phi_{-\infty} := \lim_{t \rightarrow - \infty} e^{-it\Delta}u(t) \text{ exists in } H^{s}, \text{ and } u(t) - e^{it\Delta}\phi_{-\infty} \in X^{s}_{0}(\R)  \},$$ with the norm $$\|u\|^{2}_{X^{s}(\R)} := \|\phi_{-\infty}\|^{2}_{H^{s}(\R^{n} \times \T^{d} )} + \|u - e^{it\Delta}\phi_{-\infty}\|_{ X_{0}^{s}(\R) }^{2}.$$ One can also define time-restriction spaces $X^{s}(I)$ similarly to $X_{0}^{s}(I)$. 

\subsection{The Quintic Equation on $\R \times \T$} As a first step we can transfer our Strichartz estimates to a result about $U_{\Delta}^{p}$ spaces.

\begin{lemma}\label{lem:strichartzEmbed} Let $q,p$ be as in Theorem \ref{mainThm}. Let $C$ be a cube in frequency space $\R^{n} \times \Z^{d}$ with side-length $\sim N \geq 1$ and let $I \subset \R$ be a time interval. Then $$\|\mathbbm{1}_{I} \cdot P_{C}u\|_{\ell_{\gamma}^{q}L^{p}(\R^{n}\times \T^{d} \times [\gamma, \gamma +1] )} \lesssim N^{\frac{n+d}{2} - \frac{n+d+2}{p}} \|P_{C}u\|_{U_{\Delta}^{\min(p,q)}(I, L^{2} (\R^{n} \times \T^{d} ) ) } $$
\end{lemma}

\noindent The lemma is a direct consequence of the atomic decomposition of $e^{-it\Delta}P_{C}u$. 

We begin with the case of the quintic equation on $\R \times \T$ with initial data in $H^{\frac{1}{2}}$ (recall this equation is $H^{\frac{1}{2}}$-critical). Let $F(u) = |u|^{4}u$. We will apply an iteration argument to the Duhamel operator \begin{equation} \label{eq:Duhamel} u_0 + \int_{0}^{t} e^{i(t-s)\Delta} F(u)(s) ds.\end{equation} The presentation here is similar to \cite{KV}, although we will be able to prove more in this semiperiodic setting since we have global-in-time estimates. 

A first step is the following multilinear estimate, which we will see is a corollary of Theorem \ref{mainThm} and the basic properties of the function spaces outlined above. 

\begin{lemma}\label{lem:nonlinear1} Let $I \subset \R$ be a time interval and suppose $u^{(j)} \in X^{\frac{1}{2}}(I)$ and $v \in Y^{-\frac{1}{2}}(I)$. Then \begin{equation}\label{eq:dualEst} \bigg| \int_{I} \int_{\R \times \T}  v \cdot \prod_{j=1}^{5} u^{(j)} dx dy dt \bigg| \lesssim \|v\|_{Y^{-\frac{1}{2}}} \prod_{j=1}^{5}\|u^{(j)}\|_{X^{\frac{1}{2}}(I)}. \end{equation}
\end{lemma}

\begin{proof} The proof of this lemma follows the same basic approach as in \cite{HTT} (see also \cite{KV}, \cite{HP}, and \cite{HHK}). Below all norms are taken with respect to the time interval $I$ (note we can have $I = \R$). Since we have embeddings $X^{\frac{1}{2}}\hookrightarrow Y^{\frac{1}{2}}$ it suffices to prove \eqref{eq:dualEst} with the $X^{\frac{1}{2}}$ norms replaced by $Y^{\frac{1}{2}}$ norms.
	
	We will use the short-hand $u_{N} = P_{N}u$. By performing Littlewood-Paley decompositions on all six functions and exploiting symmetry of the resulting expression, we reduce matters to showing that \begin{equation} \label{eq:dualEst2} \sum_{N_0 \geq 1} \sum_{N_1 \geq N_2 \geq ... \geq N_5} \bigg| \int_{I} \int_{\R \times \T}  v_{N_0} \cdot \prod_{j=1}^{5} u_{N_j}^{(j)} dx dy dt \bigg| \lesssim \|v\|_{Y^{-\frac{1}{2}}} \prod_{j=1}^{5}\|u^{(j)}\|_{Y^{\frac{1}{2}}(I)},\end{equation} where $N_j \geq 1$ are dyadic integers. Moreover, by Plancharel's theorem we know that for fixed $N_{0},...,N_{5}$ the corresponding term in \eqref{eq:dualEst2} vanishes unless the two largest frequencies are comparable. Hence we have two cases to consider,  when $N_{0} \sim N_{1} \geq N_2 \geq... \geq N_5$ (\textbf{Case I}), and when $N_{0} \lesssim N_1  \sim N_2 \geq ... \geq N_5$ (\textbf{Case II}).
	
	\textbf{Case I.} We fix $N_0, N_1$ with $N_0 \sim N_1$ and decompose the support of $P_{\leq N_0}$ and $P_{\leq N_1}$ into subcubes $\{C_{j}\}$ of side-lengths $N_2$. We say $C_{j} \sim C_{k}$ if the sumset $C_{j} + C_{k}$ overlaps the Fourier support of $P_{\leq 2N_2}$. After decomposing $v_{N_0}$ and $u^{(1)}_{N_1}$ with respect to the $C_{j}$'s, we see by Plancharel's theorem that it suffices to bound 
	\begin{equation} \label{eq:dualEst3} \sum_{C_{j} \sim C_{k}}\sum_{N_2 \geq N_3 \geq N_4 \geq N_5} \bigg| \int_{I} \int_{\R \times \T}  P_{C_{j}}v_{N_0} P_{C_k}u^{(1)} \cdot \prod_{j=2}^{5} u_{N_j}^{(j)} dx dy dt \bigg| \end{equation} by a factor that will be summable over $N_0$ and $N_1$ (in particular, if $C_{j}$ is not comparable to $C_{k}$ then the integral is 0). We turn to this task now. 
	
	We apply the mixed-norm H\"{o}lder's inequality with $L^{p}$ exponents corresponding to $$ \frac{4}{18} + \frac{4}{18} + \frac{4}{18} + \frac{4}{18} + \frac{1}{18} + \frac{1}{18} = 1$$ and $\ell^{q}$ exponents corresponding to $$ \frac{5}{36} + \frac{5}{36} + \frac{5}{36} + \frac{5}{36} + \frac{2}{9} + \frac{2}{9} = 1$$ to obtain  \begin{align*} \eqref{eq:dualEst3} \leq \sum_{C_{j} \sim C_{k}}\sum_{N_2 \geq N_3 \geq N_4 \geq N_5} &\|P_{C_j}v_{N_0}\|_{\ell^{\frac{36}{5}}L^{\frac{9}{2}}} \|P_{C_k}u^{(1)}_{N_1}\|_{\ell^{\frac{36}{5}}L^{\frac{9}{2}}} \|u^{(2)}_{N_2}\|_{\ell^{\frac{36}{5}}L^{\frac{9}{2}}} \|u^{(3)}_{N_3}\|_{\ell^{\frac{36}{5}}L^{\frac{9}{2}}} \\ &\cdot\|u^{(4)}_{N_4}\|_{\ell^{\frac{9}{2}}L^{18}}\|u^{(5)}_{N_5}\|_{\ell^{\frac{9}{2}}L^{18}}, \end{align*}
	
	\noindent where the norms are all localized in time to $I$. By applying Lemma \ref{lem:strichartzEmbed} and the embedding $Y^{0} \hookrightarrow U^{r}$ for $r > 2$ we ultimately obtain  \begin{align*} \eqref{eq:dualEst3} \lesssim \sum_{C_{j} \sim C_{k}}\sum_{N_2 \geq N_3 \geq N_4 \geq N_5} \frac{N_{4}^{\frac{5}{18}} N_{5}^{\frac{5}{18}} }{N_{2}^{\frac{3}{18}} N_{3}^{\frac{7}{18}} } &\|P_{C_j}v_{N_0}\|_{Y^{-\frac{1}{2}}} \|P_{C_k}u^{(1)}_{N_1}\|_{Y^{\frac{1}{2}}} \|u^{(2)}_{N_2}\|_{Y^{\frac{1}{2}}} \|u^{(3)}_{N_3}\|_{Y^{\frac{1}{2}}} \\ &\cdot\|u^{(4)}_{N_4}\|_{Y^{\frac{1}{2}}}\|u^{(5)}_{N_5}\|_{Y^{\frac{1}{2}}}. \end{align*} Now by using Cauchy-Schwarz (or Schur's test) to sum, we conclude that 
	\begin{align*}\eqref{eq:dualEst3} &\lesssim \sum_{C_{j} \sim C_{k}} \|P_{C_j}v_{N_0}\|_{Y^{-\frac{1}{2}}} \|P_{C_k}u^{(1)}_{N_1}\|_{Y^{\frac{1}{2}}} \|u^{(2)}\|_{Y^{\frac{1}{2}}} \|u^{(3)}\|_{Y^{\frac{1}{2}}}\|u^{(4)}\|_{Y^{\frac{1}{2}}}\|u^{(5)}\|_{Y^{\frac{1}{2}}} \\ &\lesssim \|v_{N_0}\|_{Y^{-\frac{1}{2}}} \|u^{(1)}_{N_1}\|_{Y^{\frac{1}{2}}}  \prod_{j=2}^{5}\|u^{(j)}\|_{Y^{\frac{1}{2}}}  \end{align*} and therefore \begin{align*}\eqref{eq:dualEst2} &\lesssim \sum_{N_0 \sim N_{1}} \|v_{N_0}\|_{Y^{-\frac{1}{2}}} \|u^{(1)}_{N_1}\|_{Y^{\frac{1}{2}}} \prod_{j=2}^{5}\|u^{(j)}\|_{Y^{\frac{1}{2}}} \\ &\lesssim \|v\|_{Y^{-\frac{1}{2}}} \prod_{j=1}^{5}\|u^{(j)}\|_{Y^{\frac{1}{2}}}, \end{align*}
	\noindent as desired. 
	
	\textbf{Case II.} In this case we do not decompose into subcubes, and instead directly apply the mixed-norm H\"{o}lder inequality with the same exponents as in Case I and then Lemma \ref{lem:strichartzEmbed}. This leads to a bound of the form 
	
	\begin{align*} \eqref{eq:dualEst2} \lesssim \sum_{N_1} \sum_{N_0 \lesssim N_1}\sum_{N_1 \sim N_2 \geq N_3 \geq N_4 \geq N_5} \frac{N_{0}^{\frac{11}{18}}N_{4}^{\frac{5}{18}}N_{5}^{\frac{5}{18}}}{N_{1}^{\frac{7}{18}}N_{2}^{\frac{7}{18}}N_{3}^{\frac{7}{18}}} &\|v_{N_0}\|_{Y^{-\frac{1}{2}}} \|u^{(1)}_{N_1}\|_{Y^{\frac{1}{2}}} \|u^{(2)}_{N_2}\|_{Y^{\frac{1}{2}}} \|u^{(3)}_{N_3}\|_{Y^{\frac{1}{2}}} \\ &\cdot\|u^{(4)}_{N_4}\|_{Y^{\frac{1}{2}}}\|u^{(5)}_{N_5}\|_{Y^{\frac{1}{2}}}. 
	\end{align*} Then repeatedly using Cauchy-Schwarz or Schur's test as before to sum, we get 
	
	\begin{align*} \eqref{eq:dualEst2} &\lesssim \sum_{N_1 \sim N_2} \frac{N_{1}^{\frac{4}{18}}}{N_{2}^{\frac{4}{18}}}  \|v\|_{Y^{-\frac{1}{2}}}\|u_{N_1}^{(1)}\|_{Y^{\frac{1}{2}}}\|u_{N_2}^{(2)}\|_{Y^{\frac{1}{2}}} \prod_{j=3}^{5}\|u^{(j)}\|_{Y^{\frac{1}{2}}} \\ &\lesssim \|v\|_{Y^{-\frac{1}{2}}} \prod_{j=1}^{5}\|u^{(j)}\|_{Y^{\frac{1}{2}}},
	\end{align*} completing the proof. 
	
\end{proof}

Below we will write $$\mathcal{I}(u) = \int_{0}^{t} e^{i(t-s)\Delta} F(u)(s) ds.$$ 

\begin{prop}\label{prop:iterationBound} For any time interval $I \subset \R$ we have \begin{equation} \label{eq:nonlinear1} \|\mathcal{I}(u)\|_{X^{\frac{1}{2}}(I)} \lesssim \|u\|_{X^{\frac{1}{2}}(I)}^{5} \end{equation} and\begin{equation} \label{eq:nonlinear2}
	\|\mathcal{I}(u+w) - \mathcal{I}(u) \|_{X^{\frac{1}{2}}(I)} \lesssim \|w\|_{X^{\frac{1}{2}}(I)}(\|u\|_{X^{\frac{1}{2}}(I)} + \|w\|_{X^{\frac{1}{2}}(I)})^{4}.
	\end{equation}
	
\end{prop} 

\begin{proof}
	
	By duality $$\| \mathcal{I}(u)\|_{X^{\frac{1}{2}}(I)} \leq \sup_{\substack{ v \in Y^{-1/2}(I) \\ \|v\|_{Y^{-1/2}} = 1  } }  \bigg| \int_{I} \int_{\R \times \T}  v \cdot F(u) dx dy dt \bigg|$$ for any $u \in X^{\frac{1}{2}}$ (the proof is similar to the proof of Proposition 2.11 in \cite{HTT}, see also \cite{HP}). We prove the second part of the proposition, since the first is then a simple consequence.
	
	It suffices to prove that $$\bigg| \int_{I} \int_{\R \times \T}  v \cdot (|u+w|^{4}(u+w) - |u|^{4}u) dx dy dt \bigg| \lesssim \|v\|_{Y^{-\frac{1}{2}}} \|w\|_{X^{\frac{1}{2}}}(\|u\|_{X^{\frac{1}{2}}} + \|w\|_{X^{\frac{1}{2}}})^{4} $$ for a fixed $v$ with $\|v\|_{Y^{-\frac{1}{2}}} = 1$. By expanding the expression inside the integral we see that this estimate is an immediate consequence of Lemma \ref{lem:nonlinear1} with $u^{(j)} \in \{ \pm u, \pm w, \pm \bar{u}, \pm \bar{w}\}.$ This completes the proof.  
	
\end{proof}

We can now prove Theorem \ref{nonlinearThm} using a standard iteration argument. As above we fix a time interval $I \subset \R$ containing 0 (in particular we could have $I = \R$). 

\subsubsection{The Global Case for Small Data} Suppose one has small initial data $u_0 \in H^{\frac{1}{2}}(\R \times \T)$ with $$\|u_0\|_{H^{\frac{1}{2}}} \leq \eta < \eta_{0},$$ with $\eta_0$ some fixed parameter to be determined. We apply a contraction mapping argument to the operator $$\Phi(u)(t) := e^{it\Delta}u_0 \pm i \mathcal{I}(u)(t)$$ on the ball $$B := \{ u \in X^{\frac{1}{2}}(I) \cap C_{t}H_{x,y}^{\frac{1}{2}}(I, \R \times \T): \|u\|_{X^{\frac{1}{2}}(I)} \leq 4\eta \}$$ with respect to the metric $d(u,v) = \|u - v \|_{X^{\frac{1}{2}}(I)}.$ By Proposition \ref{prop:iterationBound} we have \begin{align*} \|\Phi(u)\|_{X^{\frac{1}{2}}(I)} &\leq \|e^{it\Delta}u_0\|_{X^{\frac{1}{2}}(I)} + \| \mathcal{I}(u) \|_{X^{\frac{1}{2}}(I)} \\ &\leq 3\|u_0\|_{H^{\frac{1}{2}}} + C\|u\|^{5}_{X^{\frac{1}{2}}(I)} \\ &\leq 3\eta + C(4\eta)^{5} \leq 4\eta \end{align*} provided $\eta_0$ is chosen small enough. Therefore $\Phi$ maps $B$ into itself. Now the second part of Proposition \ref{prop:iterationBound} implies that 
\begin{align*} d(\Phi(u), \Phi(v)) &\lesssim \|u - v\|_{X^{\frac{1}{2}}(I)}(\|u\|_{X^{\frac{1}{2}}(I)} + \|v\|_{X^{\frac{1}{2}}(I)})^{4} \\
&\lesssim d(u,v)(8\eta)^{4} \\ &\leq \frac{1}{2}d(u,v), \end{align*} again provided $\eta_0$ is chosen sufficiently small. Now apply the contraction mapping theorem on $B$ to obtain a solution. By taking $I = \R$ we see in particular that for small initial data in $H^{\frac{1}{2}}(\R \times \T)$ one has global-in-time solutions to the quintic equation. We can additionally show that such solutions scatter in $H^{\frac{1}{2}}(\R \times \T)$ by appealing to the fact that if $G \in X^{\frac{1}{2}}(\R)$ then $\lim_{t\rightarrow \pm \infty}G(t)$ exists in $H^{\frac{1}{2}}$ (for the case $t\rightarrow \infty$ see Proposition 2.2 in \cite{HHK}). Indeed, to prove scattering it suffices to show that if $u$ is a solution to the quintic equation then the Duhamel integral $\mathcal{I}(u)(t)$ is conditionally convergent in $H^{\frac{1}{2}}$ as $t \rightarrow \pm \infty$.

We can also prove scattering directly from the argument above with a small amount of additional work. We record some details since these estimates are also useful in the local theory for large data. To facilitate the argument we introduce the following time-divisible norm as in \cite{HP}: $$\|u\|_{Z(I)} = \sum_{p = 9/2, 18 } \bigg( \sum_{N \geq 1} N^{p(\frac{4}{p} - \frac{1}{2} )} \|\mathbbm{1}_{I} P_{N}u\|_{\ell^{ q(p)}L^{p}}^{p}  \bigg)^{\frac{1}{p}}, $$ where $q(p) = \frac{4p}{p-2}$ as above. We also let $$\|u\|_{Z'(I)} = \|u\|^{\frac{3}{4}}_{Z(I)}\|u\|^{\frac{1}{4}}_{X^{\frac{1}{2}}(I)}.$$ Note that $\|u\|_{Z(I)} \lesssim \|u\|_{X^{\frac{1}{2}}(I)}$ by Lemma \ref{lem:strichartzEmbed}. By repeating the argument given in the proofs of Lemma \ref{lem:nonlinear1} and Proposition \ref{prop:iterationBound} we obtain the stronger results 
\begin{equation} \label{eq:duhamelBound} \|\mathcal{I}(u)\|_{X^{\frac{1}{2}}(I)} \lesssim \|u\|_{X^{\frac{1}{2}}}\|u\|^{4}_{Z'(I)}\end{equation} and $$\|\mathcal{I}(u+w) - \mathcal{I}(u) \|_{X^{\frac{1}{2}}(I)} \lesssim \|w\|_{X^{\frac{1}{2}}(I)}(\|u\|_{X^{\frac{1}{2}}(I)} + \|w\|_{X^{\frac{1}{2}}(I)}  )(\|u\|_{Z(I)} + \|w\|_{Z(I)})^{3}. $$  
\noindent Now if $u$ is a solution in $B$ with $\|u\|_{H^{\frac{1}{2}}} \leq \eta_0 $ then the norm $\|u\|_{Z'(I)}^{5}$ is uniformly bounded independent of $I$. Given $\epsilon > 0$, we can therefore find a time $T_{\epsilon}$ such that if $r,t > T_{\epsilon}$ then $$\|\mathcal{I}(\mathbbm{1}_{[t,r]}u)\|_{X^{\frac{1}{2}}(\R)} < \epsilon.$$ But this implies that $\mathcal{I}(u)(t) = \int_{0}^{t} e^{-is\Delta}F(u)(s) ds$ is Cauchy in $H^{\frac{1}{2}}$, and hence converges in norm as $t \rightarrow \infty$.

\subsubsection{The Local Case for Large Data} Suppose $\|u_0\|_{H^{\frac{1}{2}}} \leq A$ and fix $\delta  \geq 0$ to be determined. We apply a contraction mapping argument to the operator $\Phi$ on $$B':= \{ u \in X^{\frac{1}{2}}(I) \cap C_{t}H_{x,y}^{\frac{1}{2}}(I, \R \times \T): \|u\|_{X^{\frac{1}{2}}(I)} \leq 2A \text{ and } \|u\|_{Z(I)} \leq 2\delta\}.$$ Now if $\|e^{it\Delta}u_0\|_{Z(I)} \leq \delta$ then by \eqref{eq:duhamelBound} we see that $\Phi: B' \rightarrow B'$ if $A\delta$ is small enough, and likewise $$\|\Phi(u) - \Phi(v)\|_{X^{\frac{1}{2}}(I)} \leq \frac{1}{2} \|u-v\|_{X^{\frac{1}{2}}(I)}$$ if $A\delta$ is small enough. But for any $u_0 \in H^{\frac{1}{2}}$ we can find a time interval $I$ containing 0 such that $\|e^{it\Delta}u_0\|_{Z(I)} \leq \delta$, and this completes the argument.

\subsection{The Cubic Equation on $\R^{2} \times \T$ } The skeleton of the argument in this setting is essentially the same as what we saw in the last section, so we only provide a brief sketch of the details. The key point here is that the $\ell^{q(p)}$ exponent in dimension $(n,d) = (2,1)$ is $q(p) = \frac{4p}{2(p-2)} = \frac{2p}{p-2}.$ It follows that if $\sum_{i=0}^{3} \frac{1}{p_{i}} = 1$ then $\sum_{i=0}^{3} \frac{1}{q(p_{i})} = 1$, and therefore one has the quadrilinear estimate \begin{equation}\label{eq:cubicHolder} \bigg| \int_{I} \int_{\R^{2} \times \T} v \cdot \prod_{i=1}^{3}u^{(i)} dx dy dt\bigg| \lesssim \|v\|_{\ell_{I}^{q(p_0)} L^{p_0} }\prod_{i=1}^{3}\|u^{(i)}\|_{\ell_{I}^{q(p_i)} L^{p_i} }.
\end{equation} Note that \eqref{eq:cubicHolder} will not hold in general for the admissible exponents $r, q(r)$ on $\R \times \T^{2}$, since with these dimensions $q(r) = \frac{4r}{r-2}$. 

The main nonlinear estimate is the following.
\begin{lemma}\label{lem:nonlinearCubic} Let $I \subset \R$ be a time interval and suppose $u^{(j)} \in X^{\frac{1}{2}}(I)$ and $v \in Y^{-\frac{1}{2}}(I)$. Then \begin{equation}\label{eq:cubicDualEst} \bigg| \int_{I} \int_{\R^{2} \times \T}  v \cdot \prod_{j=1}^{3} u^{(j)} dx dy dt \bigg| \lesssim \|v\|_{Y^{-\frac{1}{2}}} \prod_{j=1}^{3}\|u^{(j)}\|_{X^{\frac{1}{2}}(I)}. \end{equation}
	
\end{lemma}

\begin{proof}
	We use the same Littlewood-Paley decomposition argument as the proof of Lemma \ref{lem:nonlinear1} with two cases corresponding to $N_0 \sim N_1 \geq N_2 \geq N_3$ (Case I), and $N_0 \lesssim N_1 \sim N_2 \geq N_3$ (Case II). In both cases we apply \eqref{eq:cubicHolder} with exponents $p_0,p_1,p_2 = \frac{7}{2}$ and $p_{3} = 7$.
	
	\textbf{Case I.} As in the proof of Lemma \ref{lem:nonlinear1} we decompose the support of $P_{N_0}$ and $P_{N_1}$ into subcubes $C_{j}$ of side length $N_2$. It then suffices to estimate \begin{equation}\label{eq:cubic1} \sum_{ N_0 \sim N_1 \geq N_2 \geq N_3} \sum_{C_j \sim C_k} \bigg| \int_{I} \int_{\R^{2} \times \T} P_{C_k} v_{N_0}P_{C_j}u_{N_1}^{(1)} \cdot u_{N_2}^{(2)}u_{N_3}^{(3)} dx dy dt \bigg|. \end{equation} By \eqref{eq:cubicHolder} and then Lemma \ref{lem:strichartzEmbed} we have 
	\begin{align*} \eqref{eq:cubic1} &\leq \sum_{ N_0 \sim N_1 \geq N_2 \geq N_3} \sum_{C_j \sim C_k} \|P_{C_k}v_{N_0}\|_{\ell^{\frac{14}{3}} L^{\frac{7}{2}} }\|P_{C_j}u_{N_1}^{(1)}\|_{\ell^{\frac{14}{3}} L^{\frac{7}{2}} }\|u_{N_2}^{(2)}\|_{\ell^{\frac{14}{3}} L^{\frac{7}{2}} }\|u_{N_3}^{(3)}\|_{\ell^{\frac{14}{5}} L^{7} } \\
	&\lesssim \sum_{ N_0 \sim N_1 \geq N_2 \geq N_3} \sum_{C_j \sim C_k} \big(\frac{N_3}{N_2}\big)^{\frac{2}{7}} \|P_{C_k}v_{N_0}\|_{Y^{-\frac{1}{2}}}\|P_{C_j}u_{N_1}^{(1)}\|_{Y^{\frac{1}{2}} }\|u_{N_2}^{(2)}\|_{Y^{\frac{1}{2}}}\|u_{N_3}^{(3)}\|_{Y^{\frac{1}{2}} }
	\end{align*} which is bounded above by the desired quantity by Cauchy-Schwarz or Schur's test. Here all spacetimes norms are taken relative to $I \times \R^{2} \times \T$.  
	
	\textbf{Case II.} In this case we estimate \begin{equation}\label{eq:cubic2} \sum_{ N_0 \lesssim N_1 \sim N_2 \geq N_3}  \bigg| \int_{I} \int_{\R^{2} \times \T} v_{N_0} u_{N_1}^{(1)} u_{N_2}^{(2)}u_{N_3}^{(3)} dx dy dt \bigg| \end{equation} by applying \eqref{eq:cubicHolder} and Lemma \ref{lem:strichartzEmbed} with the same exponents as in Case I. We get
	$$ \eqref{eq:cubic2} \lesssim  \sum_{  N_0 \lesssim N_1 \sim N_2 \geq N_3}  \frac{N_{0}^{\frac{4}{7}} N_{3}^{\frac{2}{7}} }{N_{1}^{\frac{3}{7}} N_{2}^{\frac{3}{7}} } \|v_{N_0}\|_{Y^{-\frac{1}{2}}}\|u_{N_1}^{(1)}\|_{Y^{\frac{1}{2}} }\|u_{N_2}^{(2)}\|_{Y^{\frac{1}{2}}}\|u_{N_3}^{(3)}\|_{Y^{\frac{1}{2}} } $$ which is bounded by the desired quantity (again by Cauchy-Schwarz or Schur's test). 
	
\end{proof}

The rest of the cubic case of Theorem \ref{nonlinearThm} can now be proved by routine modifications of arguments from the last section, using Lemma \ref{lem:nonlinearCubic} in place of Lemma \ref{lem:nonlinear1}.

\section{Additional Remarks}

\begin{remark} Theorem \ref{mainThm} remains true with a loss of $N^{\epsilon}$ if the operator $e^{it\Delta}$ is replaced by $$e^{it\phi(D)}f := \sum_{m} \int \widehat{f_{m}}(\xi) e^{2\pi i (x\cdot \xi + y \cdot m + t\phi(\xi, m))} d\xi,$$ where $\phi(\xi, \eta)$ is a $C^{3}$ function on $\R^{n+d}$ such that $D^{2}\phi$ is uniformly positive-definite on $\R^{n+d}$ and $$|\partial_{I}\phi(\xi,\eta)| \lesssim \frac{1}{1 + (|\xi|^{2} + |\eta|^{2} )^{\frac{1}{2}}}$$ for any triple index $I$. Let $\rho = (\xi,\eta)$ and define the rescaled function $\phi_1(\rho) = N^{-2}\phi(N\xi, N\eta).$ By Taylor expansion of the phase we see that if $\tau \subset B_{1}$ has radius $\sim N^{-\frac{2}{3}}$ then $$\int_{\tau} \widehat{g}(\rho) e^{2\pi i (z\cdot \rho + t\phi_{1}(\rho) )} d\rho$$ has space-time Fourier support in an $O(N^{-2})$-neighborhood of the paraboloid in $\R^{n+d+1}$. One then obtains a decoupling result for $e^{it\phi(D)}$ by following the iteration scheme outlined in Section 7 of \cite{BD}, and then completes the argument by following the same steps from Section 3. 

It is not clear if one can extend the $\epsilon$-removal argument from Section 4 to operators with these types of phases. The main issue is with the local Strichartz estimate in this context. The argument of Killip and Visan relies on some subtle number-theoretic properties of the kernel associated to $e^{it \Delta}$, and it is not immediately clear that the same argument will work even if $e^{it\phi(D)}$ is a small perturbation of $e^{it\Delta}$.  
\end{remark}

\vspace{2mm}

\begin{remark}\label{rmk:qsharp} The estimate in Theorem \ref{mainThm} is not true for any $q < \frac{4p}{n(p-2)}$. This is essentially a consequence of the sharpness of the $L_{t}^{q}L_{x}^{p}$ Strichartz estimates on $\R^{n}$. Indeed, we will suppose our initial data is a function independent of the periodic variables. An application of Minkowski's inequality shows that if $q \geq p$ then for any function $F(x,t)$ one has 
$$\|F\|_{L_{t}^{q} L_{x}^{p}(\R \times \R^{n}) } \leq \bigg(\sum_{\gamma \in \Z} \bigg( \int_{\R^{n}} \bigg( \int_{\gamma}^{\gamma + 1} |F(x,t)|^{q} dt \bigg)^{\frac{p}{q}} dx \bigg)^{\frac{q}{p}}  \bigg)^{\frac{1}{q}}.$$

\noindent Fix $\lambda \geq 1$ and set $F(x,t) = e^{it\Delta_{\R^n}}\phi(x)$ for a Schwartz function $\phi$ with $\widehat{\phi}$ supported inside the unit ball. Also let $f_{\lambda}$ be the function such that $\widehat{f_{\lambda}}(\xi) = \widehat{\phi}(\lambda \xi)$.  Finally let $\psi$ be a Schwartz function on $\R$ with $\psi \geq 1$ on $[0,1]$ and $\widehat{\psi}$ supported in $[-\frac{1}{2}, \frac{1}{2}]$, and set $\psi_{\gamma} (t) = \psi(t +\gamma)$. By rescaling we obtain  \begin{align*}\|e^{it \Delta_{\R^n}}\phi\|_{L_{t}^{q} L_{x}^{p} } &= \lambda^{n - \frac{n}{p} - \frac{2}{q}} \|e^{it \Delta_{\R^n}} f_{\lambda}\|_{L_{t}^{q} L_{x}^{p}} \\ &\leq  \lambda^{n - \frac{n}{p} - \frac{2}{q}}\bigg(\sum_{\gamma \in \Z} \bigg( \int_{\R^{n}} \bigg( \int_{\gamma}^{\gamma + 1} |e^{it\Delta_{\R^n} } f_{\lambda}(x)|^{q} dt \bigg)^{\frac{p}{q}} dx \bigg)^{\frac{q}{p}}  \bigg)^{\frac{1}{q}} \\ &\leq \lambda^{n - \frac{n}{p} - \frac{2}{q}}\bigg(\sum_{\gamma \in \Z} \bigg( \int_{\R^{n}} \bigg( \int_{\R} |e^{it\Delta_{\R^n} } f_{\lambda}(x) \psi_{\gamma}(t)|^q dt \bigg)^{\frac{p}{q}} dx \bigg)^{\frac{q}{p}}  \bigg)^{\frac{1}{q}}.  \end{align*} 
\noindent Now the Fourier transform in $t$ of $e^{it\Delta_{\R^n} } f_{\lambda}(x) \cdot \psi_{\gamma}(t)$ is supported in a ball of size $O(1)$ for almost every $x$, so by applying Bernstein's inequality to the $L_{t}^{q}$ integral we obtain $$\|e^{it \Delta_{\R^n}}\phi\|_{L_{t}^{q} L_{x}^{p} } \lesssim \lambda^{n - \frac{n}{p} - \frac{2}{q}}\bigg(\sum_{\gamma \in \Z} \bigg( \int_{\R^{n}} \int_{\R} |e^{it\Delta_{\R^n} } f_{\lambda}(x) \psi_{\gamma}(t)|^p dt dx \bigg)^{\frac{q}{p}}  \bigg)^{\frac{1}{q}}.$$ Hence if Theorem \ref{mainThm} is true for a choice of $q,p$ then we would need $$\|e^{it \Delta_{\R^n}}\phi \|_{L_{t}^{q} L_{x}^{p} } \lesssim \lambda^{n - \frac{n}{p} - \frac{2}{q}} \|f_{\lambda}\|_{L^{2}(\R^n )} = c\lambda^{\frac{n}{2} - \frac{n}{p} - \frac{2}{q}}\|\phi \|_{L^{2} (\R^n)}.$$ Since $\lambda \geq 1$ it follows that we need $\frac{n}{p} + \frac{2}{q} \leq \frac{n}{2},$ which is exactly the condition $q \geq \frac{4p}{n(p-2)}.$ 

\end{remark}

\vspace{2mm}

\begin{remark} As briefly discussed in the introduction,  we do not know in general if the loss of $N^{\epsilon}$ is necessary in the endpoint case $p = \frac{2(n+d+2)}{n+d}$ and $q = \frac{2(n+d+2)}{n}$ of Theorem \ref{mainThm} for arbitrary $n,d$. Note, however, that in the case $d= 0$ the theorem is true without any loss (this is just the Stein-Tomas theorem), while in the case $n = 0$ some loss in $N$ is necessary (as shown by Bourgain \cite{B}). 
	
In the low-dimensional case $n = d = 1$ we do know that the local Strichartz estimate holds with no loss in $N$ for $p = 4$, as shown by Takaoka and Tzvetkov \cite{TT}. However, as far as we know there are no higher-dimensional analogues of the Takaoka-Tzvetkov result beyond what is present in this paper. As a first step to understanding the full endpoint case of Theorem \ref{mainThm} one would hope to be able to prove a local $L^{2}$ to $L^{\frac{2(n+d+2)}{n+d}}$ Strichartz estimate or to understand why such an estimate fails for larger $(n,d)$. Note that the Takaoka-Tzvetkov argument relies heavily on the fact that $p = 4$ is an even exponent, along with a geometric estimate that is special to $\R \times \Z$. So new ideas are likely needed.

It is also of interest to better understand the possible ways to partition the time interval in global norms of the type \eqref{eq:mainThmEst}, and to understand how a local estimate of the type $$\|e^{it \Delta_{\R^n \times \T^d} } f \|_{L^{p}(\R^n \times \T^d \times I) }\leq C_{I}\|f\|_{H^s (\R^n \times \T^d)}$$ depends on the length of $I$. As far as we know this is only fully understood in the case $d = 0$ and in the special case where $n = 0$ and $\T^d$ is not an irrational torus; see \cite{DGG} for some recent related work.  

 \end{remark}

\Addresses
\end{document}